\documentclass[12pt]{article}

\usepackage{amsmath,amssymb,latexsym, amsfonts, amscd, amsthm, mathrsfs, multirow}

\theoremstyle{plain}

\newtheorem{ithm}{Theorem}
\newtheorem{iconj}{Conjecture}
\newtheorem{theorem}{Theorem}[section]

\newtheorem{proposition}[theorem]{Proposition}

\newtheorem{lemma}[theorem]{Lemma}

\newtheorem{oldthm}{Theorem}

\theoremstyle{definition}
\newtheorem{definition}[theorem]{Definition}
\newtheorem{remark}[theorem]{Remark}

\long\def\symbolfootnote[#1]#2{\begingroup
\def\thefootnote{\fnsymbol{footnote}}\footnote[#1]{#2}\endgroup}

\def\lra{{\longrightarrow}}

\def\GL{{\bf GL}}

\def\X{\mathbf{X}}
\def\O{\mathbf{O}}

\def\cA{\mathcal{A}}
\def\sgn{\mathrm{sgn}}
\def\N{\mathrm{N}}
\def\1{\mf{1}}

\DeclareMathOperator{\Hom}{Hom} 

\DeclareMathOperator{\sign}{sign}

\DeclareMathOperator{\Tr}{Tr}
\DeclareMathOperator{\ord}{ord}

\DeclareMathOperator\Ind{Ind}

\newcommand{\Res}{\mathrm{\,Res\,}}

\newcommand{\stack}[2]{\genfrac{}{}{0pt}{}{#1}{#2}}

\newcommand{\mf}{\mathfrak }

\def\l{\ell}
\def\fa{\mathfrak{a}}

\def\fp{\mathfrak{p}}
\def\fq{\mathfrak{q}}
\def\fg{\mathfrak{g}}

\def\fZ{\mathfrak{Z}}
\def\fL{\mathfrak{L}}

\def\Z{\mathbf{Z}}
\def\F{\mathbf{F}}
\def\Q{\mathbf{Q}}
\def\C{\mathbf{C}}

\def\R{\mathbf{R}}

\def\bdf{\begin{defn}}
\def\edf{\end{defn}}
\def\cM{\mathcal{M}}

\def\cD{\mathbf{D}}
\def\cO{\mathcal{O}}

\def\cV{\mathcal{V}}
\def\cW{\mathcal{W}}

\def\ff{\mathfrak{f}}
\def\fb{\mathfrak{b}}
\def\fc{\mathfrak{c}}

\def\Gal{{\rm Gal}}

\def\sP{\mathcal{P}}
\def\sQ{\mathcal{Q}}

\def\mbf{\mathbf}
\def\mb{\mathbb}

\def\cF{{\cal F}}

\def\sA{{A}}
\def\sL{L}

\def\sB{\mathbf{B}}

\begin{document}
\title{Integral Eisenstein cocycles on $\GL_n$, I: \\  Sczech's cocycle and $p$-adic $L$-functions of \\
Totally Real Fields}

\author{Pierre Charollois \\ Samit Dasgupta\thanks{We would like to thank Michael Spiess for several discussions regarding his formalism for $p$-adic $L$-functions and for openly sharing his ideas that are presented in Section~\ref{s:oov}.   We would also like to thank Robert Sczech and Glenn Stevens for their suggestions and encouragement, and for their papers \cite{Sc2} and 
\cite{st} that helped inspire this work.
Samit Dasgupta   was partially supported by NSF grants DMS 0900924 and DMS 0952251 (CAREER), as well as a fellowship from the Sloan Foundation. Part of this paper was written while Pierre Charollois enjoyed  the hospitality of UC Santa Cruz.}}

\maketitle

\begin{abstract}
We define an integral version of Sczech's Eisenstein cocycle on $\GL_n$ by smoothing at a prime $\ell.$  As a result we obtain a new proof of the integrality of the values at nonpositive integers of the smoothed partial zeta functions  associated to ray class extensions of totally real fields.
We also obtain
a new construction of the $p$-adic $L$-functions associated to these extensions.  
Our cohomological construction allows for a study of the leading term of 
these $p$-adic $L$-functions at $s=0$.
We apply Spiess's formalism  to prove that the order of vanishing at $s=0$  is at least equal to the expected one, as conjectured by Gross.  This result was already known from Wiles' proof of the Iwasawa Main Conjecture. 
\end{abstract}

 \tableofcontents

\section*{Introduction}

Let $F$ be a totally real field of degree $n$, and let $\ff$ be an integral ideal of $F$.
For a fractional ideal $\fa$ of $F$ relatively prime to the conductor $\ff$, consider the partial zeta function
\begin{equation} \label{e:zetadef}
 \zeta_{\ff}(\fa, s) = \sum_{\fb \sim \fa} \frac{1}{\N\fb^s}, \qquad \text{Re}(s) > 1. 
 \end{equation}
Here the sum ranges over integral ideals $\fb \subset F$ equivalent to $\fa$ in the narrow ray class group modulo $\ff$, which we denote $G_\ff$.
A classical result of Siegel and Klingen states that the partial zeta functions $\zeta_{\ff}(\fa, s)$, which may be extended to 
meromorphic functions on the complex plane,  assume rational values at nonpositive integers $s$. 
Siegel proved this fact by realizing these special values as the constant terms of certain
Eisenstein series on the Hilbert modular group associated to $F$.  The rationality of the constant
terms follows from the rationality of the other Fourier coefficients, which have a simple form.

Shintani gave an alternate proof of the Siegel--Klingen result using a ``geometry of numbers" approach.
Shintani fixed an isomorphism $F \otimes_\Q \R \cong \R^n$, and considered a fundamental domain $D$ for the
action of the group of totally positive units in $F$ congruent to 1 modulo $\ff$ on the totally positive orthant of $\R^n$.  The partial zeta functions
of $F$ could then be expressed as a sum indexed by the points of $D$ contained in various lattices in $\R^n$.
Shintani evaluated these sums using standard techniques from complex analysis and
expressed them explicitly in terms of sums of products of Bernoulli polynomials.

In 1993, Sczech gave yet another proof of the Siegel--Klingen rationality theorem.  
He defined an ``Eisenstein" cocycle $\Psi$ on $\GL_n(\Q)$ valued in a  space of $\Q$-valued distributions denoted $M_\Q$.  He then showed that the cohomology class $[\Psi] \in H^{n-1}(\GL_n(\Q), M_\Q)$ could be paired with certain classes
in a dual homology group to
 yield the special values of
all totally real fields $F$ of degree $n$ at nonpositive integers, thereby demonstrating their rationality.

Each of these proofs of the Siegel--Klingen rationality theorem bears an integral refinement.  Deligne and Ribet gave such a refinement of Siegel's method following an idea initiated by Serre.
They constructed a model over $\Z$
of the relevant Hilbert modular scheme, and proved that its fibers in characteristic $p$ are geometrically irreducible. 
Meanwhile,  Barsky and Pi.~Cassou-Nogu\`es  proved an integral refinement of Shintani's formulas and interpreted these results in terms of $p$-adic measures.

\bigskip

The first goal of the present paper is to provide an integral refinement of Sczech's cocycle $\Psi$.  We introduce a ``smoothing" operation with respect to a prime $\ell$, and use it to define a cocycle $\Psi_\ell$ that satisfies an important integrality property (see Theorem~\ref{t:main} below for a precise statement).
As an application of our results, we give new proofs of the following two celebrated theorems of Deligne--Ribet and Cassou-Nogu\`es.

\begin{ithm} \label{t:integral} 
 Let $\fc$ be an integral ideal of $F$ relatively prime to $\ff$ with prime norm $\ell$. The smoothed zeta function
\begin{equation}\label{e:zetacdef}
 \zeta_{\ff, \fc}(\fa, s) = \zeta_{\ff}(\fa\fc, s) - \N\fc^{1-s}\zeta_\ff(\fa, s)
\end{equation}
assumes values in $\Z[1/\ell]$ at nonpositive integers $s$.\symbolfootnote[2]{See Remark~\ref{r:candl} for a discussion of the condition that $\fc$ has prime norm.  Also, we note that Theorem~\ref{t:integral} has the corollary that
the ``twice smoothed" zeta function
\begin{equation*}
\begin{split}
  \zeta_{\ff, \fc, \fb}(\fa, s) & = \zeta_{\ff, \fb}(\fa\fc, s) - \N\fc^{1-s}\zeta_{\ff, \fb}(\fa, s) \\
  & = \zeta_{\ff, \fc}(\fa\fb, s) - \N\fb^{1-s}\zeta_{\ff, \fc}(\fa, s)
\end{split}
\end{equation*}
assumes {\em integer} values at nonpositive integers $s$ when $(\N\fb, \N\fc) = 1$.  Gross has recently provided
an interpretation of these integers in terms of dimensions of certain spaces of automorphic forms \cite{gross2}.
}
\end{ithm}
Our integrality results further allow for a new construction of the Deligne--Ribet--Cassou-Nogu\`es $p$-adic zeta functions $\zeta_{\ff, \fc, p}(\fa, s)$ interpolating the classical zeta values $\zeta_{\ff, \fc}(\fa,  s)$.  
Define $\zeta_{\ff}^*(\fa,  s)$ as in (\ref{e:zetadef}), but with the sum restricted to ideals $\fb$ relatively prime to $p$; define $\zeta_{\ff, \fc}^*(\fa,  s)$ from $\zeta_{\ff}^*(\fa,  s)$ as in (\ref{e:zetacdef}).
Let $\cW$ denote the {\em weight space} of continuous homomorphisms from $\Z_p^*$ to $\C_p^*$, with $k \in \Z$ embedded as $x \mapsto x^k.$

\begin{ithm} \label{t:padic} 
 Let $\fc$ be an integral ideal of $F$ relatively prime to $\ff p$ with prime norm $\ell$.
 There exists a unique $\Z_p$-valued analytic function
$\zeta_{\ff, \fc, p}(\fa, s)$ of the  variable $s \in \cW$ such that
\[ \zeta_{\ff, \fc, p}(\fa,  - k) = \zeta^*_{\ff, \fc}(\fa,  -k)
\]
for all nonnegative integers $k$.
\end{ithm}

\bigskip
Our  construction of the $p$-adic zeta functions of totally real fields allows us to embark on a new study 
of the behavior of the leading terms of these functions at $s=0$.
In order to state our main result, it is convenient to work with the $p$-adic $L$-functions associated to characters rather than the $p$-adic zeta functions associated to ideal classes.
To this end,  let $\chi: \Gal(\overline{F}/F) \rightarrow \overline{\Q}^*$ be a totally odd finite order character with conductor $\ff$.  We fix embeddings $\overline{\Q} \hookrightarrow \C$ and $\overline{\Q} \hookrightarrow \overline{\Q}_p$, so that $\chi$ can be viewed as taking values in $\C$ or $\overline{\Q}_p$.  Let\symbolfootnote[2]{As usual, replace $\mu_{p-1}$ by $\{\pm 1\}$ when $p=2$.}  \[ \omega: \Gal(\overline{F}/F) \longrightarrow \mu_{p-1} \subset \overline{\Q}^* \] denote the Teichm\"uller character.  There is a $p$-adic $L$-function  $L_{\fc,p}(\chi\omega, s)\colon \Z_p \lra \C_p^*$ associated to the totally even character $\chi \omega$, given by
\begin{equation}
 L_{\fc, p}(\chi \omega, s) :=  \sum_{\fa \in G_\ff} \chi(\fa\fc) \zeta_{\ff, \fc, p}(\fa, \langle \cdot \rangle^s ) \label{e:lpdef}
 \end{equation}
  where $\langle x \rangle = x/\omega(x)$ for $x \in \Z_p^*$.\symbolfootnote[3]{We are now following the classical conventions
  and viewing $L_{\fc, p}$ as a function of $\Z_p$ rather than as a function of weight space $\cW$.  An element 
  $s \in \Z_p$ yields the element $x \mapsto \langle x \rangle^s$ of $\cW$ appearing on the right side of (\ref{e:lpdef}).}
It satisfies the interpolation property
 \begin{align*}
  L_{\fc, p}(\chi \omega, -k) &= L_{\fc}^*(\chi \omega^{-k}, -k)  \\
  & := L^*(\chi \omega^{-k}, -k)(1 - \chi\omega^{-k}(\fc)\N\fc^{1+k})
  \end{align*}
 for integers $k \ge 0$, where $L^*(\chi, s)$ denotes the  classical $L$-function with Euler factors at the primes dividing $p$ removed.

 Let $r_\chi$ denote the number of primes $\fp$ of $F$ above $p$ such that $\chi(\fp) = 1$.  It is well-known that
\[ \ord_{s=0} L_{\fc}^*(\chi, s) = \ord_{s=0} L^*(\chi, s) =  r_\chi \]
  (see \cite[2.6]{tate}).  
In \cite{gross}, Gross proposed the following:
\begin{iconj}[Gross] \label{c:gross} We have
 \[ \ord_{s=0} L_{\fc, p}(\chi\omega, s) = r_\chi. \]
\end{iconj}
Combining our cohomological construction of the $p$-adic $L$-function with Spiess's formalism (see \S5), we prove the following partial result towards Gross's conjecture:
\begin{ithm} \label{t:oov}
We have
 \[ \ord_{s=0} L_{\fc, p}(\chi\omega, s) \ge r_\chi. \]
\end{ithm}

The result of Theorem~\ref{t:oov} was already known from Wiles' proof of the Iwasawa Main Conjecture
under the auxiliary assumption \begin{equation*}\label{eqWilesS} F_\infty \cap H = F, \end{equation*}where $H$ denotes the fixed field of $\chi$ and
 $F_\infty$ denotes the cyclotomic $\Z_p$-extension of $F$.  
This assumption 
(namely, that $\chi$ has ``type S'' in Greenberg's terminology) was removed by Snaith in [Sn, Theorem 6.2.5] using Brauer induction (see also [BG, pp.\,  165-166] for further discussion). 

Our method contrasts with that of Wiles in that it is purely analytic; we calculate the first $r_\chi - 1$ derivatives of $L_{\fc, p}(\chi\omega, s)$ at $s=0$ directly and show that they vanish.
Spiess proved Theorem~\ref{t:oov} as well using his formalism \cite{spiess}.  His cohomology classes are defined using Shintani's method.  In  \cite{CDG} we provide a direct comparison between the cocycles defined using the methods of Sczech and Shintani.

 \bigskip

We conclude the introduction by recalling Sczech's method, which is central to this article, and describing our integral refinement.
Let $\sP = \Q[X_1, \ldots, X_n]$. 
We denote by $\sQ$ a certain space of tuples of linear forms on $\R^n$, defined precisely in
Section~\ref{s:cocycledef}.  The spaces $\sP$ and $\sQ$ are endowed with a left action of $\Gamma = \GL_n(\Q)$ given by $(AP)(X)=P(XA).$
Let $M_\Q$ denote the $\Q$-vector space of functions  
\[ 
\phi : \sP \times \sQ \times  (\Q/\Z)^n \longrightarrow \Q \]
that are $\Q$-linear in the first variable and satisfy the distribution relation
\begin{equation}
\label{e:dist}
  \phi(P, Q,v)  = \sgn(\lambda)^n \sum_{\stack{w\in (\Q/\Z)^n}{\lambda w=v}}\phi( \lambda^{\deg P} P, \lambda^{-1} Q,w) 
  \end{equation}
for all nonzero integers $\lambda$, when $P \in \sP$ is homogeneous.   

We view the elements of $(\Q/\Z)^n$ as column vectors and define a left $\Gamma$-action on $M_\Q$ as follows.
Given $\gamma \in \Gamma$, choose a nonzero scalar multiple $A = \lambda \gamma$ with 
$\lambda \in \Z$ such that $A \in M_n(\Z)$.  For $f \in M_\Q$, define
\begin{equation} \label{e:vaction}
 (\gamma f)(P, Q, v) = \sgn(\det(A)) \sum_{r \in \Z^n / A \Z^n} f(A^t P, A^{-1} Q , A^{-1}(r + v)). 
 \end{equation}
 The distribution relation (\ref{e:dist}) implies that (\ref{e:vaction}) does not depend on the auxiliary choice of $\lambda$. 
 
Sczech defined a homogeneous cocycle 
\[ \Psi \in Z^{n-1}(\Gamma, M_\Q) \subset C^{n-1}(\Gamma, M_\Q) = \Hom_\Gamma(\Z[\Gamma^n], M_\Q)  \] called the Eisenstein cocycle, representing a class \[ [\Psi] \in H^{n-1}(\Gamma, M_\Q).\]  We recall this definition precisely in Section~\ref{s:cocycle}.
 The values at nonpositive integers of the zeta functions of all totally real fields $F$ of degree $n$ can be obtained from
certain specializations of $[\Psi]$
as follows.  

Fix a conductor $\ff$ and an integral ideal $\fa$ as above.  
Associated to $F, \ff, \fa$, and a nonnegative integer $k$, we define a certain homology class
\[ [\fZ_{\fa, \ff, k}] \in H_{n-1}(\Gamma, M_\Q^\vee), \]
where $M_\Q$ denotes the $\Q$-linear dual of $M_\Q$.  Fix a $\Z$-basis $\{w_1, \dots, w_n\}$ 
for  $\fa^{-1}\ff$.  Let $P \in \Z[X_1, \dots, X_n]$ denote the homogeneous polynomial of degree $n$ given up to scalar by
the norm:
\begin{equation} \label{e:Pdef}
 P(X_1, \dotsc, X_n) = \N(\fa)\N(w_1X_1 + \cdots + w_nX_n).
 \end{equation}
Let $Q = (Q_1, \dotsc, Q_n)$ be the $n$-tuple of linear forms given by
\begin{equation} \label{e:Qdef} Q_i = \tau_i(w_1^*)X_1 + \cdots + \tau_i(w_n^*)X_n,  \end{equation}
where $\{w_1^*, \dotsc, w_n^*\}$ denotes the dual basis  with respect to the trace form on $F$, and 
the $\tau_i$ for $i=1, \dotsc, n$ denote the embeddings $F \hookrightarrow \R$.
Let \begin{equation} \label{e:vdef} v = (\Tr(w_1^*), \dotsc, \Tr(w_n^*)). \end{equation}
Denote by $f_k = f_{P, Q, v}$ the element of $M_\Q^\vee$ defined by evaluation at $(P^k, Q, v)$:
\[ f_k(\phi) := \phi(P^k, Q, v). \]

Finally, we define elements $A_1, \dotsc, A_{n-1} \in \Gamma$ by considering 
 the action of a basis of totally positive units of $F$ congruent to 1 modulo $\ff$ via multiplication on 
$\fa^{-1}\ff$, in terms of the basis $\{w_i\}$ for $\fa^{-1}\ff$.  
Homogenizing and symmetrizing the tuple $(A_1, \dots, A_{n-1})$ yields
a certain homogeneous $(n-1)$-chain $\sA \in \Z[\Gamma^n]$ (see~(\ref{e:adef}) for the precise formula).

We then define
\[ \fZ_{\fa, \ff, k} = \sA \otimes f_k \in C_{n-1}(\Gamma, M_\Q^\vee) = \Z[\Gamma^n] \otimes_\Gamma M_\Q^\vee.  \]
Using the definition of $\sA$ in (\ref{e:adef}), it follows from the fact that the $A_i$ commute and that $f_k$ is invariant under the $A_i$ that $\fZ_{\fa, \ff, k}$ is in fact an $(n-1)$-cycle.  The homology class that it represents depends only on $\ff$, $\fa$, and $k$, and not on any other choices made.

The cap product yields a canonical pairing 
\[
 \langle \cdot, \cdot \rangle: H^{n-1}(\Gamma, M_\Q)  \times H_{n-1}(\Gamma, M_\Q^\vee) \longrightarrow \Q \]
 given by $\langle [\Psi], [\sA \otimes f] \rangle = f(\Psi(\sA))$ and extended by linearity.
 Sczech proved the formula
\begin{equation} \label{e:special}
 \zeta_{\ff}(\fa,  - k) = \langle [\Psi], [\fZ_{\fa, \ff, k}] \rangle  \in \Q
\end{equation}
for integers $k \ge 0$, thereby completing his proof of the Siegel--Klingen rationality theorem.


\bigskip

In this paper, we fix a prime $\ell$ and consider the congruence subgroup $\Gamma_{\ell} \subset \Gamma \cap \GL_n(\Z_\ell)$
consisting of matrices whose first column has all elements but the first divisible by $\ell$.  We define a cocycle  
$\Psi_\ell^+ \in Z^{n-1}(\Gamma_{\ell}, M_\Q)$ derived from Sczech's $\Psi$ by smoothing at the prime $\ell$ (see (\ref{e:psielldef}) for a precise formula).  
We also define a cocycle $\Psi_\ell$ refining $\Psi_\ell^+$ from which $\Psi_\ell^+$ can be recovered by projecting on to the subspace invariant under the action of multiplication by $-1$ on $\sQ$.
  Our key result is the following.
  
\begin{ithm}  \label{t:main}
The cocycle $\Psi_\ell$ takes values in  the $\Z[\frac{1}{\ell}][\Gamma_{\ell}]$-submodule $M_\ell \subset M_\Q$
consisting of distributions $\phi$ such that $\phi(P, Q, v) \in \frac{1}{m}\Z[\frac{1}{\ell}]$ when $P \in \Z[\frac{1}{\ell}][X_1, \dots, X_n]$ is homogeneous and has the property
\begin{equation} \label{e:pcond}
 P(v + {\textstyle\frac{1}{\ell}}\Z \oplus \Z^{n-1}) \subset \Z[{\textstyle\frac{1}{\ell}}],
 \end{equation}
and $Q \in \sQ$ is an $m$-tuple of linear forms.
 The cocycle $\Psi_\ell^+$ takes values in $\frac{1}{2} M_\ell$.
\end{ithm}
  In particular, taking $P = 1$, we find that $\Psi_\ell(\sA, 1, Q, v) \in \frac{1}{m}\Z[1/\ell]$ for all $\sA, Q,$ and $v$.
This integrality property of our cocycle $\Psi_\ell$ lies in sharp contrast to Sczech's $\Psi$, which assumes fractional values with $p$-adically unbounded denominator for each $p$
as $v$ varies (with $P=1$ and fixed $\sA, Q$).

Now let $F$ be as above, and let $\fc$ be a prime ideal of $F$ with norm $\ell$.
  Using Sczech's formula (\ref{e:special}), we prove that certain specializations of the class $\Psi_\ell^+$ yield the smoothed partial zeta functions of $F$ at nonpositive integers:
\begin{equation} \label{e:intscz}
 \zeta_{\ff, \fc}(\fa, -k) = \langle [\Psi_\ell^+], [\fZ_{\fa, \ff, k, \ell}] \rangle 
\end{equation}
where $[\fZ_{\fa, \ff, k, \ell}] \in H_{n-1}(\Gamma_{\ell}, M_\ell^\vee)$ is defined similarly to $\fZ_{\fa, \ff, k}$, but slightly modified to account for the $\ell$-smoothing (see~(\ref{e:Pdef2})--(\ref{e:adef}) in Section~\ref{s:integral} for the precise definition).
Here $M_\ell^\vee$ denotes the $\Z[\frac{1}{\ell}]$-dual of $M_\ell$.
Since $H^{n-1}(\Gamma_{\ell}, M_\ell)$ and $H_{n-1}(\Gamma_{\ell}, M_\ell^\vee)$ pair to $\Z[\frac{1}{\ell}]$,
we arrive at our proof of Theorem~\ref{t:integral}.\footnote{In fact, since the cocycle $\Psi_\ell^+$ defined from Sczech's method takes values in $\frac{1}{2} M_\ell$, and since the $Q$ defined in (\ref{e:Qdef}) and used by Sczech is an $n$-tuple, the results of this paper prove that the value $\zeta_{\ff, \fc}(\fa, s)$ lies 
in $\frac{1}{2n} \Z[1/\ell]$ for nonpositive integers $s$.  The factor $1/(2n)$ can be eliminated by considering the refined cocycle $\Psi_\ell$ and proving that one need only consider an individual linear form $Q_i$ rather than the entire tuple $Q$.  These aspects are studied in \cite{CDG}; see Remark~\ref{r:cdg} and Section~\ref{s:integral} below for further details.}

 We prove Theorem~\ref{t:padic}
by interpreting the cocycle $\Psi_\ell$ in 
terms of $p$-adic measures and thereby defining a certain measure-valued cocycle $\mu_\ell$. 
Using the cohomology class $[\mu_\ell]$ and applying equation (\ref{e:intscz}), we 
construct the $p$-adic zeta functions $\zeta_{\ff, \fc, p}(\fa, s)$ with the desired interpolation property. 
Theorem~\ref{t:oov} is proven by using our cohomological construction of the $p$-adic $L$-function to recognize the 
values $L_{\fc, p}^{(k)}(\chi \omega, 0)$ for nonnegative integers $k$ as the pairing of $[\mu_\ell]$ with
certain homology classes denoted $[\fL_k]$, and applying results of Spiess (Theorems~\ref{t:sp1} and~\ref{t:sp2}) that implies that these homology classes vanish for $k < r_\chi$.

\bigskip

The fact that the $p$-adic zeta functions of all totally real fields of degree $n$ that contain a prime of norm $\ell$ arise as the specializations of a single cohomology class $[\Psi_\ell]$ is striking.  
It seems promising to study these $p$-adic zeta functions by taking certain {\em other} specializations of the same class. 
 As an example of this phenomenon, we hope to study Gross's conjectural formula for the leading term $L_{\fc, p}^{(r_\chi)}(\chi \omega, 0)$ in future work, building on our prior investigations (\cite{darmon-dasgupta}, \cite{dasgupta}).
Also in future work (\cite{CDG}), we will demonstrate how to provide an alternate construction of the Eisenstein class $[\Psi_\ell]$ using Shintani's method, following prior works by Solomon, Hill, Colmez, and the second author.
 
This paper is organized as follows.  In Section~\ref{s:cocycle}, we give the precise definition of Sczech's cocycle $\Psi$ 
and recall the formula for $\Psi$ in terms of Dedekind sums derived in \cite{GS}.   In Section~\ref{s:smoothed} we  define our smoothed cocycle $\Psi_\ell^+$ and its refinement $\Psi_\ell$.
We prove Theorem~\ref{t:main}, the key integrality result concerning $\Psi_\ell$ and the technical heart of the paper.  In Section~\ref{s:integral} we combine Theorem~\ref{t:main} with a suitable generalization of Sczech's formula (\ref{e:special}) to prove Theorem~\ref{t:integral}.  In Section~\ref{s:padic}, we interpret our construction in terms of $p$-adic measures, and thereby prove Theorem~\ref{t:padic}.  We conclude in Section~\ref{s:oov} by proving Theorem~\ref{t:oov} using Spiess' results.

\section{The Eisenstein cocycle} \label{s:cocycle}

The goal of this section is to recall the definition of Sczech's cocycle $\Psi$ and some of its salient properties.

\subsection{Sczech's Eisenstein cocycle for $\GL_n(\Q)$} \label{s:cocycledef}

We begin by recalling Sczech's  Eisenstein cocycle for $\Gamma = \GL_n(\Q)$. See \cite{Sc2}, \cite[Section 6]{GS}, or \cite{CGS}
 Section 2 for a more detailed exposition.

\medskip
Let $\sA =(A_1,\ldots,A_n)\in \Gamma^n$ be an $n$-tuple of matrices.
 Fix $x \in \R^n - \{0\}$.  For each matrix $A_i$, let $\sigma_i$ denote the first (i.e.\ leftmost) column of $A_i$ such that 
 $\langle x,\sigma_i \rangle \neq 0$.
Denote by $\sigma = (\sigma_{ij})$ the square matrix with columns  $\sigma_i$, for $1\leq i \leq n$, and define
\[
\psi_{\sA}(x)=\frac{\det(\sigma)}{\langle x,\sigma_1\rangle \cdots\langle x,\sigma_n\rangle }.
\]

More generally, for any homogeneous polynomial $P(X_1, \ldots, X_n)$, we consider the partial differential operator $P(-\partial_{x_1}, \ldots, -\partial_{x_n})$ and   define the function
\begin{align}\psi_{\sA}(P, x)=&\,P(-\partial_{x_1}, \ldots, -\partial_{x_n})
\psi_{\sA}(x) \nonumber \\
=&\, \det(\sigma)\sum_r P_r(\sigma)\prod_{j=1}^{n} \frac {1} {\langle x,\sigma_j\rangle^{1+r_j} } \label{e:psipdef},
\end{align}
where $r=(r_1,\dotsc, r_n)$ runs over all partitions of $\deg(P)$ into nonnegative integers $r_j,$
and $P_r(\sigma)$ is the homogeneous polynomial in the $\sigma_{ij}$ satisfying  the relation 
\begin{equation} \label{e:prdef}
P(X\sigma^t)=\sum_r P_r(\sigma)\frac {X_1^{r_1}\cdots X_n^{r_n}}{r_1!\cdots r_n !}.
\end{equation}

The Eisenstein cocycle is essentially given by summing the value of $\psi_{\sA}(P, x)$ over all $x \in \Z^n - \{0\}$:
\begin{equation} \label{e:fakepsi}
 \text{ `` } \Psi(\sA,P, v)=(2\pi i)^{-n-\deg(P)}\sum_{x \in \Z^n - \{0\}} \textbf{e}(\langle x,v\rangle)\psi_{\sA}(P, x). \text{ '' }
\end{equation}
The sum in (\ref{e:fakepsi}) converges only conditionally, so to make sense of it one uses Sczech's $Q$-summation trick.
To this end we fix a family of $m$  linear  forms $Q_1, \dotsc, Q_m$
on $\R^n$ such that each form $Q_i$ is nonvanishing on $\Q^n-\{0\}.$ Let $\sQ$ denote the set of such $m$-tuples $Q = (Q_1, \dotsc, Q_m)$ of linear forms. 
 We view each $Q_i$ as a row vector, and for any row vector $x\in \R^n$ we
 adopt the notation 
 \begin{equation} \label{defQ}
 Q_i(x)= Q_i x^t =\sum_{j=1}^n Q_{ij} x_j, \qquad
Q(x)=\prod_{i=1}^m Q_i(x).\end{equation}
We can identify $Q$ with an $m\times n$  matrix with real rows $Q_i.$
The set $\sQ$ is endowed with a left action of $\Gamma$
described in terms of matrices by $AQ \leftrightarrow (Q_{ij})A^t$.  The action with respect to the corresponding functions on row vectors
is given by $AQ(x) = Q(xA)$.

Given   $Q \in \sQ,$ and a sequence $a(x)$ indexed by a subset $L$ of a lattice in $ \R^n$, the $Q$-summation of $a(x)$ over $L$ is defined by
\begin{equation}
\sum_{x\in L} a(x)\!\! \mid_Q \ =\lim_{t\rightarrow \infty} \sum_{\stack{x\in L}{|Q(x)| <t}} \!\! a(x),
\label{e:qsumlimit}
\end{equation}
under the assumption that the sum of $a(x)$ over the $x \in L$ such that $|Q(x)| < t$ converges absolutely for all $t$, and that the limit in (\ref{e:qsumlimit}) exists.
Sczech proved that the $Q$-summation
\begin{equation}\label{defPsi}\Psi(\sA,P, Q,v)=(2\pi i)^{-n-\deg(P)}\sum_{x\in \Z^n - \{0\}} \textbf{e}(\langle x,v\rangle)\psi_{\sA}(P, x)
\! \mid_Q \end{equation}
exists.


Let us recall some notation from the Introduction.
Define the $\Q[\Gamma]$-module $M_{\Q}$ to be the space of functions  
\[ 
\phi :  \sP \times \sQ \times  \Q^n/\Z^n \longrightarrow \Q \]
that are $\Q$-linear in the first variable and satisfy the distribution relation
\[
  \phi(P, Q,v) = \sgn(\lambda)^n \sum_{\stack{w\in (\Q/\Z)^n}{\lambda w=v}}\phi( \lambda^{\deg P} P, \lambda^{-1} Q,w) 
\]
for all nonzero integers $\lambda$, when $P \in \sP$ is homogeneous. The $\Gamma$-action on $M_\Q$ is given in equation (\ref{e:vaction}).
  Sczech proved:

\begin{theorem} \label{theoremSczechcocycle}The  map  $\sA\mapsto \Psi(\sA,\cdot , \cdot, \cdot)$ defines a $M_\Q$-valued  homogenous  $(n-1)$-cocycle on $\Gamma$.
    Moreover, it represents a non-trivial cohomology class in $H^{n-1}(\Gamma, M_\Q)$.
  \end{theorem}

This is  \cite[Theorem 4]{Sc2}, restricted to $u =0$.

\subsection{Decomposition} \label{s:decomp}

For future calculations, it is convenient to decompose the sum in (\ref{defPsi}) according to the various matrices $\sigma$ that may occur.
To this end, for each $d = (d_1, \dotsc, d_n) \in \Z^n$ such that $1 \le d_i \le n$, let $\sigma(d)$ denote the $n \times n$ matrix 
whose $i$th column is the $d_i$th column of $A_i$.  Let $X(A, d) = X(d) \subset \R^n - \{0\}$ denote the set of $x$ whose associated 
matrix $\sigma$ is equal to $\sigma(d)$, i.e.\ such that the first column of $A_i$ not orthogonal to $x$ is the $d_i$th, for each
$i = 1, \dotsc, n$.

Write $\1 = (1, 1, \dotsc, 1)$, and for an $n$-tuple $e = (e_1, \dotsc, e_n)$, write $\underline{e} = \sum_{i=1}^{n} e_i$.
Gathering the terms together according to the finite partition $\{X(d)\}_d$ of $\R^n-\{0\},$ we  obtain
\begin{equation}\label{PsiPS}
\Psi(\sA,P, Q,v)=(-1)^n\sum_{d} \sgn(\det(\sigma(d)) \sum_r \frac{P_r(\sigma)}{(\1 + r)!} \cD^+(X(d) \cap \Z^n,\sigma(d) ,\mf{1}+r,Q, v),
\end{equation}
where 
\begin{equation}\label{defD}
\cD^+(L,\sigma,e, Q, v) =\frac {(-1)^n |\det(\sigma)| e!} {(2\pi i)^{\underline{e}} } \sum_{x\in  L} 
\frac {\textbf{e}(\langle x,v\rangle)}{\langle x, \sigma_1\rangle^{e_1}\cdots \langle x, \sigma_n\rangle^{e_n}}{|}_Q.
\end{equation}
We have introduced the factor $ (-1)^n |\det(\sigma)| e!$ in (\ref{defD}) to simplify future calculations.

\subsection{Bernoulli distributions} \label{s:bern}

The sums in (\ref{defD}) are Dedekind sums that can be evaluated in terms of Bernoulli distributions,
whose definition we now recall.
For each integer $k \ge 0$, the Bernoulli polynomial $b_k(x)$ is defined by 
the generating function 
\[ \frac{te^{xt}}{e^t - 1} = \sum_{k=0}^{\infty} b_k(x)\frac{t^k}{k!}. \]
For $k \neq 1$, we define the periodic Bernoulli function $B_k(x) = b_k(\{x\})$, where
$\{x\} \in [0, 1)$ denotes the fractional part of $x$.
For $k=1$, we reconcile the discrepancy between $b_1(0)$ and $b_1(1)$ by defining
\[ B_1(x) = \begin{cases} 
b_1(\{x\}) = \{x\} - 1/2 & x \not \in \Z \\
0 & x \in \Z. 
\end{cases}
\]

Given an $n$-tuple of positive integers $e=(e_1,\dotsc,e_n)$ and an element $x \in \R^n$,  
define
$$
\sB_e(x)=\prod_{j=1}^n B_{e_j}(x_j).
$$ 
Writing $\underline{e} = \sum e_j$, the function
$\sB_e$ provides a $\Q$-valued  distribution on $(\Q/\Z)^n$ of weight $\underline{e} - n$,
in the sense that for each integer $N$, we have
\begin{equation}
\sB_e(x)=N^{\underline{e}-n} \!\!\!\! \sum_{y\in (\Z/N\Z)^n} \sB_e\left(\frac{x +y}{N}\right).
\end{equation}
\medskip

To relate the distributions $\sB_e$  to Sczech's cocycle $\Psi$, they must be altered by a defect 
arising from the $Q$-limit summation process discussed in Section~\ref{s:cocycledef}.

\begin{definition}
Let $e=(e_1,\ldots, e_n)$ be a vector of positive integers, and $v\in (\Q/\Z)^n.$
Let 
\begin{equation}\label{defJ}
J=\{1\leq j\leq n \, \mid  e_j=1  \textrm{ and } v_j\in \Z\}.\end{equation}
Define
$$\sB_e(v,Q) = \frac 1 m \sum_{i=1}^m \left(\prod_{j\in J} \frac{\sign Q_{ij}} 2 \right)  \prod _{j\notin J}  B_{e_j}(v_j).$$
(In particular, if $J$ is empty, then  $\sB_e(v,Q)= \sB_e(v)$ does not depend on $Q$.) Define
\[\sB_e^+(v, Q) = \frac{1}{2}(\sB_e(v, Q) + \sB_e(v, -Q)) = \begin{cases} \sB_e(v, Q) & \text{ if } \#J \text{ is even,} \\
0 & \text{ otherwise.}
\end{cases} \]
\end{definition}

The functions $\sB_e(\ \! .\  \!,Q)$ and $\sB_e^+(\ \! .\  \!,Q)$
are distributions on $(\Q/\Z)^n$ of weight $\underline{e} - n$.

\begin{proposition}[\cite{GS}, Proposition 2.7] \label{p:evald} Let $e$ be an $n$-tuple of positive integers, $Q \in \sQ$, $v \in (\Q/\Z)^n$,  and $\sigma \in M_n(\Z)$.
Let \begin{equation} L = \{x \in \Z^n: \langle x, \sigma_i \rangle \neq 0 \text{ for } 1 \le i \le n \}. \label{ldef} \end{equation}
We have
\begin{equation} \label{e:dedsum}
  \cD^+(L, \sigma, e, Q, v) = 
\sum_{x\in  \Z^n/\sigma \Z^n }  \sB_e^+(\sigma^{-1}(x+v), \sigma^{-1}Q), \end{equation}
where the right side is understood to have the value 0 when $\det(\sigma) = 0$.
\end{proposition}

Proposition~\ref{p:evald} gives the value of the Dedekind sum appearing in (\ref{PsiPS}) for
$d = \1$; this will be sufficient for our applications.  Whenever $L$ is given in terms of $\sigma$ as in (\ref{ldef}), we drop it from our notation and write simply $\cD^+(\sigma, e, v,Q)$.

\section{The smoothed cocycle}
\label{s:smoothed}

Let $\l$ be a prime number, and let $\Z_{(\ell)} = \Z[1/p, p \neq \ell]$ denote the localization of $\Z$ at the prime ideal $(\ell)$.
Our aim in this section  is  to smooth the cocycle $\Psi$ at the prime $\ell$, yielding a cocycle $\Psi_\l^+$ defined on the congruence subgroup 
$$ \Gamma_\ell := \Gamma_0(\l \Z_{(\ell)}) =\{A\in \GL_n(\Z_{(\ell)}): 
A \equiv \left(\begin{array}{ccc} *  & *&*  \\ 0  & *&*\\ \vdots  & \vdots & \vdots  \\ 0  &* & *\end{array}\right) \!\! \mod \l\}.$$
We will then prove Theorem~\ref{t:main}, an integrality result for the smoothed cocycle $\Psi_\ell^+$ and a refinement $\Psi_\ell$.

\subsection{Definition of the smoothed cocycle}

Consider the diagonal matrix whose first entry is $\ell$ and other diagonal entries are equal to $1$:
$$\pi_\l =\left(\begin{array}{cccc}\ell & &  &\\
 &1&  & \\
 & & \ddots &\\
&  & &1\end{array} \right).$$ 
 For $\sA=(A_1,\ldots, A_n)\in \Gamma_{\ell}^n$, let \[ \sA'= \pi_\l \sA \pi_\ell^{-1}  = (\pi_\l A_1\pi_\l^{-1}, \ldots ,\pi_\l A_n\pi_\l^{-1})  \in \GL_n(\Z_{(\ell)})^n. \] 
 
  Fix a homogeneous polynomial $P\in \sP$ and define $P' = \pi_\l^{-1}P.$ 
It is a straightforward computation to check that the coefficients of $P$ and $P'$ defined in (\ref{e:prdef}) satisfy
\begin{equation}
 P'_r(\sigma'(d))=P_r(\sigma(d)) \cdot \ell^{-\Sigma(r, d)} \qquad \text{ where } \qquad \Sigma(r, d) = \sum_{i: \, d_i = 1} r_i. \label{e:pequal}
 \end{equation}
 Here as in Section~\ref{s:decomp}, $d = (d_1, \dotsc, d_n)$ is a tuple of integers with $1 \le d_i \le n$, and $\sigma(d)$ denotes the square matrix whose $i$th column is the $d_i$th column of $A_i$.  The matrix $\sigma'(d)$ denotes the similarly constructed matrix with $A_i$ replaced by $A_i'$.
 
Next fix  $ Q\in \sQ$ as in (\ref{defQ}), and let  $Q'=\pi_\l Q$.  For $v \in (\Q/\Z)^n$, let $v' =\pi_\ell v$. We define 
 \begin{equation} \label{e:psielldef}
  \Psi_\ell^+(\sA, P, Q, v) = \Psi(\sA',P', Q', v')-\ell \Psi( \sA, P, Q, v). 
  \end{equation}
The following is a straightforward computation using the fact that $\Psi$ is a $\Gamma$-cocycle. 
 

\begin{proposition}  \label{propoDPsi}
The function $\Psi_\ell^+$ is a homogeneous $(n-1)$-cocycle on $\Gamma_{\ell}$ valued in $M_\Q$:
\[ \Psi_\ell^+ \in Z^{n-1}(\Gamma_{\ell}, M_\Q). \]
%
%
\end{proposition}


\subsection{A Dedekind sum formula}

In this section we prove  that
in the analogue of the Dedekind sum formula (\ref{PsiPS}) for $\Psi_\l$, all terms other than those arising from $d = \mf 1$ cancel.

In what follows, all objects associated to $\sA'$ instead of $\sA$ will be denoted with a ``prime", such as $X'(d)=X(\sA', d).$ 

\begin{lemma}\label{compareXtoX'}
For $d\neq \mf 1,$ the  map \[  \pi_\l   : (x_1,x_2, \ldots, x_n)\mapsto (\ell x_1, x_2,\ldots,x_n) \] induces a bijection between
  $X'(d) \cap  \Z^n$ and $X(d) \cap \Z^n.$ \end{lemma}

\begin{proof}
Denoting the $j$th column of the matrix $A_i$ by $A_{ij}$, note that
for $x\in \mb \R^n$, we have $\langle x, A_{i j}' \rangle = 0 \Leftrightarrow  \langle \pi_\ell x, A_{i j} \rangle = 0$.
In particular, $\pi_\ell$ gives a bijection between   $X'(d)$ and $X(d)$, and hence induces an injection from 
 $X'(d) \cap  \Z^n$ to $X(d) \cap \Z^n.$
 
To show that the map is surjective, we make use of the assumption  $d\neq \mf{1}$ to obtain an index $i$, say $i=1$, for which $d_i>1.$ For any given $z\in X(d) \cap \Z^n,$ the condition $\langle z,A_{11}\rangle=0$  ensures that its coordinates   satisfy  
  $$a_{11} z_1+\ldots +a_{n1}z_n=0.$$ This equation implies  $a_{11}z_1\equiv 0 \pmod \l$  since  $A_1 \in \Gamma_{\ell}.$ Moreover, $a_{11}$ is coprime to $\ell$ since $\det(A_1) \in \Z_{(\ell)}^*.$ We conclude that $\ell$ divides $z_1,$ hence  $z$ is of the form $ \pi_\l x$ for some $x\in \Z^n$
  as desired.
  \end{proof}

Next we use  Lemma \ref{compareXtoX'} to compare $\psi_{\sA'}$ and $\psi_{\sA}$ on the sets
 $X'(d) \cap  \Z^n$ and $X(d) \cap \Z^n$, respectively.
A direct computation shows that for $x\in X'(d),$   one has

 \begin{equation} \label{e:dotprod}
  \langle x, \sigma'_i(d)\rangle= 
 \begin{cases}
  \langle \pi_\l x, \sigma_i(d)\rangle & \text{if } d_i>1, \\
\ell^{-1} \langle \pi_\l x, \sigma_i(d)\rangle &  \text{if } d_i=1.
\end{cases}
\end{equation}
On the other hand, \begin{equation} \label{e:det}
\det(\sigma'(d))= \ell^{1-\#\{i:  \ d_i=1\}}\det(\sigma(d)). \end{equation}
Equations (\ref{e:dotprod}) and (\ref{e:det})  yield   
  \begin{equation} \label{e:detequal}
  \frac{\det(\sigma '(d))}{ \prod_{i=1}^{n} \langle x,\sigma'_i(d)\rangle^{1+r_i}}
 =  \frac{\ell^{1 + \Sigma(r, d)} \det(\sigma(d))}{\prod_{i=1}^n \langle \pi_\ell  x,\sigma_i(d)\rangle^{1 + r_i}}
 \end{equation}
 for $x\in X'(d)$, where $\Sigma(r, d)$ is defined as in (\ref{e:pequal}).
Multiplying (\ref{e:detequal}) by (\ref{e:pequal}) and recalling the definition of $\psi$ given in
 (\ref{e:psipdef}), we obtain
\[  \psi_{\sA'}(P', x)  =  
\l \psi_{\sA}(P, \pi_\l x).
\]
Multiplying by
\[ \textbf{e}(\langle x,\pi_\ell v\rangle)
= \textbf{e}(\langle \pi_\ell x, v\rangle) \]
and taking the $Q'$-summation
 over all $x\in X'(d) \cap \Z^n$,  Lemma \ref{compareXtoX'}  
 implies that in the evaluation of 
 \[ \Psi_\ell^+(\sA, P, Q, v) =  \Psi(\sA', P', Q',\pi_\ell v)- \ell \Psi( \sA,P, Q, v), \]
 the terms in (\ref{PsiPS}) for $d \neq \1$ cancel.  
 
 We have therefore proven 
 the following explicit formula for $\Psi_\ell^+$.
 Let $\sA \in \Gamma_\l^n$.  Scale the matrices in $A$ by an integer relatively prime to $\ell$ so that they all have integer entries; the formulae below will be independent of the chosen scaling integer by the distribution relation for Bernoulli polynomials.
  The matrix $\sigma=\sigma(\mf 1)$ has all rows but the first divisible by $\l$ (and all entries in the first row relatively prime to $\l$).  Therefore, the matrix
$\sigma_\l ={\pi_\l}\ell^{-1} \sigma$
has entries in $\Z_{(\ell)}$.
Define the $\l$-smoothed Dedekind sum:
\begin{equation}
\begin{split}
\cD_\l^+(\sigma, e,Q, v) = & \   {} \cD^+(\sigma_\ell, e,  \pi_\ell Q, \pi_\ell v)-\l^{1-n+\underline{e}} \cD^+(\sigma,e,  Q, v). \\
=& {}
 \sum_{x'\in \Z ^n/\sigma_\l  \Z^n }  \sB_e^+(\sigma_\l ^{-1}(x'+  \pi_\ell v), \sigma^{-1}Q ) \\
 & {} \ \ \ \ \ -\l^{1-n+\underline{e}} \!\!\!\!\! \sum_{x\in  \Z ^n/\sigma  \Z^n }  \sB_e^+(\sigma^{-1}( x+ v ),\sigma^{-1}Q). \label{e:dldef}
\end{split}
\end{equation}

 \begin{proposition} \label{p:psiellded} We have
\[
 \Psi_\ell^+(\sA, P, Q, v) = (-1)^n\sgn(\det \sigma)  \sum_r \frac{P_r(\sigma)}{(r+\1)! \ell^{\underline{r}} } \cD_\l^+(\sigma, \mf 1 + r, v, Q).
\]
\end{proposition}

\subsection{A refined cocycle} \label{s:refined}

Define $\cD_\ell$ as in (\ref{e:dldef}) with $\sB^+$ replaced with $\sB$.
Define $\Psi_\ell$ as in Proposition~\ref{p:psiellded}, with $\cD_\l^+$ replaced by $\cD_\ell$:
\begin{equation} \label{e:psiellref}
 \Psi_\ell(\sA, P, Q, v) := (-1)^n\sgn(\det \sigma)  \sum_r \frac{P_r(\sigma)}{(r+\1)! \ell^{\underline{r}} } \cD_\l(\sigma, \mf 1 + r,  Q, v).
\end{equation}
Then $\Psi_\ell^+$ is recovered from $\Psi_\ell$ by projection onto the $+1$ eigenspace of the action of multiplication by $-1$ on $\sQ$:
\[  \Psi_\ell^+(\sA, P, Q, v) = \frac{1}{2}( \Psi_\ell(\sA, P, Q, v) +  \Psi_\ell(\sA, P, -Q, v)). \]

In \cite{CDG}, we show that the result of Proposition~\ref{propoDPsi} holds still for $\Psi_\ell$:

\begin{proposition}  \label{propoDPsi2}
The function $\Psi_\ell$ is a homogeneous $(n-1)$-cocycle on $\Gamma_{\ell}$ valued in $M_\Q$:
\[ \Psi_\ell \in Z^{n-1}(\Gamma_{\ell}, M_\Q). \]
\end{proposition}

\begin{remark}  The $\Gamma_\ell$-invariance of  $\Psi_\ell$ (i.e.\ that $\Psi_\ell$ is a homogeneous cochain) is easy to 
check directly from the definition.  However, the alternating property of $\Psi_\ell$ (i.e.\ that $\Psi_\ell$ is a cocycle) is mysterious
using our explicit definition. In $\cite{CDG}$, we show that for $P=1$, the function $\Psi_\ell$ is equal to a cocycle defined using 
Shintani's method; the case for general $P$ then follows from Theorem~\ref{t:mint} below.
\end{remark}

\subsection{Decomposition of the Dedekind Sum}
\label{sectiongeneralsigma}

Consider the following $\Z[\frac{1}{\ell}][\Gamma]$-submodule of $M_\Q$:
\begin{equation} \label{e:mldef}
\begin{split} M_\ell = & \ \{ \phi \in M_\Q:  \phi(P, Q, v) \in {\textstyle \frac{1}{m}}\Z[\textstyle{\frac{1}{\ell}}] \text{ when }    P \in \Z[\frac{1}{\ell}][X_1, \dots, X_n] \\ & 
 \ \text{ is homogeneous and } P(v + (\textstyle{\frac{1}{\ell}}\Z \oplus \Z^{n-1})) \subset \Z[\frac{1}{\ell}]\}.
\end{split}
\end{equation}
In (\ref{e:mldef}), the constant $m$ denotes the number of linear forms defining the element $Q \in \sQ$, as in Section~\ref{s:cocycledef}.
In the remainder of this section, we prove Theorem~\ref{t:main}, which states that the smoothed cocycle $\Psi_\ell$ takes values in $M_\ell$.

\begin{remark} \label{r:cdg}
The cocycle $\Psi_\ell$ was introduced in Section~\ref{s:refined} because it takes values in $M_\ell$.  The smoothed Sczech
$\Psi_\ell^+$ takes values in $\frac{1}{2} M_\ell$.  Furthermore, we prove in \cite{CDG} that the cocycle $\Psi_\ell$ may be paired with
an appropriate cycle involving $Q \in \sQ$ with $m=1$ to yield the special values of zeta functions, thereby leading to a proof of Theorem~\ref{t:integral}
(see Section~\ref{s:integral}).  Sczech's formulas relating $\Psi$ to zeta values involves a $Q$ with $m = n = [F:\Q]$.  This leads to a 
slightly weaker version of Theorem~\ref{t:integral}, where the zeta values are shown to lie in $\frac{1}{2n} \Z[1/\ell]$.
\end{remark}

In view of the definition (\ref{e:psiellref}), our first step in proving the integrality of $\Psi_\ell$ is to
decompose $\cD_\l(\sigma, e, v, Q)$ into a sum of terms that individually share an analogous integrality property.
To this end,  fix a linear map $\sL\in \Hom(\Z^n,\F_\ell),$ and  assume that its values on the standard basis of $\Z^n$ are all
nonzero. For a tuple $e=(e_1, \ldots, e_n)$ of positive integers, $x \in \Q^n$, and $z\in \F_\ell$ define
\begin{equation}
\label{defB_eell}\sB_e^{\sL,z}(x,Q) =
\sB_e(x, Q)-\ell^{1-n+\underline{e}} \!\!\! \sum_{\stack{y\in \F_{\!\ell}^{\, n}}{\sL(y)=z}} \sB_e\left(\frac{x+y} \l, Q\right),\end{equation} where
the summation  runs  over all $y \in \F_{\!\l} ^{\, n}$  such that  $\sL(y)=z$.
Note that $\sB_e^{\sL,z}$ depends on $x$ mod $\ell\Z^n$ rather than mod $\Z^n$, since the summation over $y$ is restricted.
It satisfies the following distribution relation for integers $N$ relatively prime to $\ell$:
\begin{equation} \label{e:restricteddist}
  \sB_e^{\sL, Nz}(x, Q) = N^{\underline{e}-n} \!\!\!\sum_{k \in (\ell\Z/\ell N\Z)^n} \sB_e^{\sL, z}\left(\frac{x + k}{N}, Q \right).
\end{equation}

As in the previous section, consider $\sA \in \Gamma_{\ell}^n$ and $\sigma = \sigma(\mf 1)$ the square matrix consisting of the 
first columns of the matrices in the tuple $\sA$.  Recall $\sigma_\ell = \pi_\ell \ell^{-1} \sigma$.
For $y \in \Z^n/\ell\Z^n$, define $\sL(y) = \langle R, y \rangle \pmod{\ell}$,
where $R$ denotes the first row of $\sigma$.  Our desired decomposition is :

\begin{lemma}  \label{l:decompose}
Let $\{x = (x_1, \dots, x_n)\} \subset \Z^n$ denote a set of representatives for $\Z^n/\sigma_\l \Z^n$.  We have
\begin{equation} \label{e:decompose}
 \cD_\ell(\sigma, e, v, Q) = \sum_{x} \sB_e^{\sL,-x_1}(\sigma_\l ^{-1}(x+ \pi_\ell v), \sigma^{-1}Q).
 \end{equation}
\end{lemma}
  
 \begin{remark}
One easily checks that the summand in (\ref{e:decompose}) is independent of the choice of representative $x \in \Z^n$ for each class 
in $\Z^n/\sigma_\l \Z^n$.
 \end{remark}
  
  Before proving the lemma, we state an elementary lemma that will be used to change variables of indices in double sums several times
  in this article.
  
  \begin{lemma} \label{l:bijlem}
Let $A, B, C, D$ be finite index subgroups of an abelian group $E$ such that $A \supset B$ and $C \supset D$.
Let $\{x_i\}_{i=1}^{m}$ and $\{y_j\}_{j=1}^{n}$ be sets of representatives for $A/B$ and $C/D$, respectively.  Suppose that
$\#(E/D) = mn$ and that $A \cap C \subset B$.  Then the map \[ \psi: A/B \times C/D \rightarrow E/D\ \] given by
$\psi( \overline{x}_i, \overline{y}_j) = \overline{x_i + y_j}$ is a bijection of sets.
  \end{lemma}
  
  \begin{remark}  Note that the map $\psi$ depends on the  set of representatives $\{x_i\}$ chosen  and is hence only a bijection of sets and not a homomorphism of groups.
  \end{remark}
  
  \begin{remark}
Note that the conditions in the third sentence of  Lemma~\ref{l:bijlem} are automatically satisfied if $A = E$ and $B=C$.
  \end{remark}
  
  \begin{proof}[Proof of Lemma~\ref{l:decompose}.]  Applying Lemma~\ref{l:bijlem} with  $A = E = \pi_\ell^{-1} \Z^n$, $B = C =  \ell^{-1} \sigma \Z^n$, and 
  $D = \sigma \Z^n$, we see that the map $(x, y) \mapsto z = \pi_\ell^{-1} x + \l^{-1}\sigma y$ induces  a bijection between $\Z^n/\sigma_\l \Z ^n \times \Z^n/\ell\Z^n$ and $\pi_\ell^{-1} \Z^n/\sigma \Z^n$.  (We stress again that this map is only a bijection of sets dependent on the representatives $x$ chosen for $\Z^n/\sigma_\l \Z ^n$; it is not a group homomorphism.)
Under this bijection we have \[ \sL(y) \equiv  -x_1 \!\!\!\!\pmod{\l}  \Longleftrightarrow z \in \Z^n. \]  The result follows immediately from the definitions (\ref{e:dldef}) and (\ref{defB_eell}).
  \end{proof}
  
  From (\ref{e:psiellref}) and Lemma~\ref{l:decompose}, we obtain:
\begin{proposition} \label{p:keypsiform} We have
\[  \Psi_\ell(\sA, P, Q, v) = \pm \sum_{r} \frac{P_r(\sigma)}{(r+ \1)! \ell^{\underline{r}}} \!\sum_{x \in \Z^n/\sigma_\ell \Z^n} \sB_{1+r}^{\sL, -x_1}(\sigma_\ell^{-1}(x + \pi_\ell v), \sigma^{-1}Q), \]
where the $\pm$ sign is given by $(-1)^n \sgn(\det \sigma)$.
\end{proposition}

In the next two subsections, we show that the individual terms in (\ref{e:decompose}) are integral when $e = \1$, thereby proving the integrality property for $\Psi_\ell$ when $P=1$.  The integrality of $\Psi_\ell$  in general will follow by bootstrapping from this base case.
  

\subsection{A cyclotomic Dedekind sum}

The following ``cyclotomic Dedekind sum" attached to a real number $x$ will play an important role in our computations:


  \begin{equation}
  B_1^{\exp}(x,r)=\sum_{m=1}^\ell  \textbf{e}\!\left(\frac{rm}\ell\right)B_1\!\left(\frac{x+m}\ell\right)
  \end{equation}
  for any $x \in \R$ and $r \in \F_\ell^\times.$

  \begin{lemma}\label{lemmaChi} The value of  the cyclotomic Dedekind sum is given by
  \begin{equation}B_1^{\exp}(x,r )=\frac{ \mathbf{ e }(\frac {-r[x]} \l )}{\mathbf{e}(\frac r \l)-1} +\frac{\delta_{x}}{2} \mathbf{e}\left(-\frac{rx}{\ell}\right),\end{equation}
where $\delta_{x} = 1$ if $x \in \Z$ and is 0 otherwise.
  \end{lemma}
  \begin{proof}
  Translating  $m$ by the integer $[x]$ gives $m'=m+[x],$ which runs  over the same set of classes mod $\ell.$ This  yields
  $$B_1^{\exp}(x,r )=\sum_{m'\!\!\!\!\!\mod \!\l} \textbf{e}\left(\frac{r(m'-[x])} \l\right)B_1\left(\frac {m'+\{x\}} \ell\right), $$ where 
   $0\leq \{x\}< 1$ is the fractional part of $x.$ Choosing representatives $m'$ between $0$ and $\l-1$ gives
  \[ B_1\left(\frac {m'+\{x\}} {\ell}\right)=\frac {m'+\{x\}} {\ell} -\frac{1}{2} \]
  unless $m' = \{x\} = 0$, in which case the $-1/2$ term does not occur.
  Therefore
  \begin{equation}\label{B1exp}B_1^{\exp}(x,r)=\textbf{e}\left(\frac{-r[x]}\l\right) \sum_{m'=0}^{\l-1} \frac {m'}\l \textbf{e}\left(\frac{rm'}\l\right) + \frac{\delta_{x}}{2} \mathbf{e}\left(\frac{-rx}{\ell}\right) ,\end{equation}
  since the terms coming from $\{x\}$ and $\frac 12$ each vanish in the sum over $m'.$
  The value of the sum in (\ref{B1exp}) can be  obtained by evaluating  at $z=r$ the derivative  of

  $$\frac 1 {2 \pi i }\sum_{m'= 0}^{\l-1}\textbf{e}\left(\frac {m'z}\l\right)=\frac 1 {2\pi i} \cdot \frac{\textbf{e}(z)-1}{\textbf{e}(z/\l) -1}.$$
  The lemma follows.
\end{proof}

\subsection{The case $e = \1$}
\label{subsectionsimple}

The goal of this section is to bound the denominators of the individual terms in (\ref{e:decompose}) when $e = \1$. 

 \begin{proposition}\label{p:simplecase} Let  $x\in  \Q^n$. The quantity $\sB_\1^{\sL,z}(x,Q)$ lies in $\frac{1}{m}\Z[1/\ell]$, and lies in $\frac{1}{m}\Z$ if $\ell > n+1$. 
\end{proposition}

\begin{proof} This proof follows the argument of \cite[\S 6.1]{dasgupta}.
We begin by   relaxing the restricted  summation. Since the map 
$$y\mapsto \frac 1 \l \sum_{k=0}^{\l-1} \textbf{e}\left(\frac {k\sL(y)} \l\right)$$ is 
 the characteristic function of the kernel
of $\sL$, we obtain 
\begin{equation}
\label{equatB}\sB_e^{\sL,z}(x,Q)=- \ell^{\underline{e} - n}\sum_{k=1}^{\l-1} \sum_{y\in \mathbf F_{\!\l}^{\,n}}\textbf{e}\left(\frac{k(\sL(y)-z)}\l\right) \sB_e\left(\frac{x+y} \l,Q\right). \end{equation}
Note that the term $k=0$ cancels the leading term of $\sB_e^{\sL,z}$
using the distribution relation for $\sB_e$.

Let  $a_j$ be the
value of the linear form $\sL$ on the $j$-th term of the standard
basis of $\Z^n.$ We specialize to $e = \1$ and consider the case where no component of
$x$ is an integer (so there is no dependence on $Q$).
The sum (\ref{equatB}) then decomposes as
$$-\sum_{k=1}^{\l-1} \mathbf{e}\left(-\frac{kz} \l\right)\sum_{y_1=1}^\l \ldots \sum_{y_n=1}^{\l}  \prod_{j=1}^n \textbf{e}\left(\frac{ka_jy_j}\l\right)B_1\left(\frac{x_j+y_j} \l \right).$$
Since each $a_j$ is non-zero by assumption, we
can 
repeatedly  use Lemma  \ref{lemmaChi} to obtain
\begin{align}
\sB_{\1}^{\sL,z}(x,Q)
=&\,
-\sum_{k=1}^{\ell-1}\mathbf{e}(-\frac{kz}\l)\left( \prod_{j=1}^n \frac {\mathbf{e}(-\frac{ka_j[x_j]} \l  )}{\textbf{e}(\frac {ka_j} \l )-1}\right) \nonumber \\
=&\, - \Tr_{\Q(\zeta_\l )/\Q}\left(\frac {  \mathbf{e}(\frac
{-z-\sL([x])} \l )  } {\prod_{j=1}^n(\textbf{e}(\frac {a_j} \l )-1)}\right). \label{B1integral}
\end{align} 
For any primitive $\l$-th root
of unity $\zeta_\ell$, the element $\zeta_\ell - 1$ is supported at $\ell$ and has $\ell$-adic valuation $1/(\ell - 1)$.
Therefore, the  expression (\ref{B1integral}) lies in $\Z[\frac 1 \l]$ and 
has denominator at most $\l^\frac n{\l-1}.$ The lemma is proven in the case where no component $x_j$ is integral.

For the general case, let $J_0 = \{j: x_j \in \Z\}$.  An argument as above yields the following value for $\sB_\1^{\sL,z}(x,Q)$: 
\begin{equation}\label{B1integralQtrue} 
-\frac {1}{m2^{\#J_0}} \sum_{i=1}^m \sum_{J \subset J_0} 
\Tr_{\Q(\zeta_\l )/\Q}\left(\frac {  \textbf{e}(\frac
{-z-\sL([x])} \l )}{\prod_{j \not \in J_0} (\textbf{e}(\frac{a_j}{\ell}) - 1)} \prod_{j \in J_0 - J} \frac{\textbf{e}(\frac{a_j}{\ell}) + 1  } {\textbf{e}(\frac {a_j} \l )-1}\prod_{j\in J} \sign Q_{ij}\right).
\end{equation}
As in (\ref{B1integral}), the sum in (\ref{B1integralQtrue}) lies in $\Z[1/\ell]$ and in fact lies in $\Z$ if $\ell > n+1$.  It only remains to eliminate the factor $1/2^{\#J_0}$.  When summed over all $J \subset J_0$, the argument of the trace in (\ref{B1integralQtrue}) is equal to a constant depending on $J_0$ (but
not $J$) times
\begin{align}
 \sum_{J \subset J_0} \left(\prod_{j \in J_0 - J}   \frac{\textbf{e}(\frac{a_j}{\ell}) + 1  } {\textbf{e}(\frac {a_j} \l )-1}\prod_{j\in J} \sign Q_{ij} \right) &= \prod_{j \in J_0} \left(   \frac{\textbf{e}(\frac{a_j}{\ell}) + 1  } {\textbf{e}(\frac {a_j} \l )-1} + \sign Q_{ij}\right) \label{e:insidetr} \\
 &= 2^{\# J_0} \prod_{j \in J_0} \left(  \frac{ \alpha_{ij} } {\textbf{e}(\frac {a_j} \l )-1}\right), \label{e:got2j}
\end{align}
where $\alpha_{ij} = \textbf{e}(a_j/\ell)$ or $\alpha_{ij} = 1$ in the cases $\sign Q_{ij} = 1$ or $\sign Q_{ij} = -1$, respectively.
The factor $2^{\#J_0}$ in (\ref{e:got2j}) cancels that in (\ref{B1integralQtrue}).
\end{proof}

\subsection{Proof of Theorem~\ref{t:main}} \label{s:main}

We are now ready to complete the proof of Theorem~\ref{t:main} from the introduction, which states that
 the rational number $\Psi_\ell(\sA, P, Q, v)$ lies in $\frac{1}{m} \Z[1/\ell]$ when (\ref{e:pcond}) is satisfied.  We do this by showing that it lies in $\frac{1}{m}\Z_p$ for each prime $p\neq \ell$.

\begin{proposition}\label{p:crucial}
Let  $x\in \Q^n$ and  let $p\neq \l $  be a prime number.  Let $r$ be an $n$-tuple of nonnegative integers.
There exists an integer $\epsilon$ depending only on $r$, $\ell$ and the denominator of $x$, such that for all positive integers $M$ we have the following congruence between rational numbers:

\begin{equation}\label{crucialcongruence}
p^{M \cdot {\underline r}} \N(r + \1)^{-1}\sB_{\1+r}^{\sL,z} \left(\frac x {p^M},Q\right)\equiv\sB_\1^{\sL,z} \left(\frac x {p^M},Q\right)
 \N x^r \mod p^{M-\epsilon} \Z_p.
\end{equation}
\begin{remark}\label{remx=r=0} Here $\N x^r$ is shorthand for $\prod_{j=1}^{n} x_j^{r_j}$.
If it happens that both $r_j$ and $x_j$ are zero,  the corresponding term $x_j^{r_j}$ is understood to equal  $1.$  
\end{remark}
\end{proposition}

Before proving Proposition~\ref{p:crucial}, we show how it enables the proof of Theorem~\ref{t:main}.

 \begin{proof}[Proof of Theorem~\ref{t:main}]
 We recall Proposition~\ref{p:keypsiform}:
 \begin{equation} \label{e:psifind}
 \Psi_\ell(\sA, P, Q, v) = \pm \sum_{r} \frac{P_r(\sigma)}{r! \ell^{\underline{r}}} \N(r+1)^{-1} \!\!\!\!\!\!\sum_{x \in \Z^n/\sigma_\ell \Z^n} \sB_{\1+r}^{\sL, -x_1}(\sigma_\ell^{-1}(x + \pi_\ell v), \sigma^{-1}Q).
\end{equation}
For each $x$ in the sum above we let $y = \sigma_\ell^{-1}(x + \pi_\ell v)$ and note that $y$ has the property 
\begin{equation} \label{e:yprop}
 \textstyle{\frac{1}{\ell}} \sigma(y) \in v + \textstyle{\frac{1}{\ell}} \Z \oplus \Z^{n-1}.
\end{equation} Fix a prime $p \neq \ell$.
For each $y$ we let $\epsilon$ be as in Proposition~\ref{p:crucial} and fix a positive integer $M \ge \epsilon$.  Applying the distribution relation (\ref{e:restricteddist}) we replace the term
$\sB_{1+r}^{\sL, -x_1}(y, \sigma^{-1}Q)$ in (\ref{e:psifind}) with
\begin{equation} \label{e:usedist}
 p^{M \cdot {\underline r}} \sum_{k \in (\ell\Z/\ell p^M \Z)^n} \sB^{\sL, z}_{\1 + r}\left(\frac{y + k}{p^M}, \sigma^{-1}Q \right),  
 \end{equation}
where $z \equiv -p^{-M}x_1 \pmod{\ell}$.  By Proposition~\ref{p:crucial} and the choice of $M$, the quantity in (\ref{e:usedist}) multiplied by $\N(r + 1)^{-1}$ is congruent modulo $\Z_p$ to 
\[  \sum_{k \in (\ell\Z/\ell p^M \Z)^n} \sB_\1^{\sL, z}\left(\frac{y + k}{p^M}, \sigma^{-1}Q \right) \N\!\left(y+ k\right)^{r}
\]
Plugging this expression into (\ref{e:psifind}), we note that each coefficient $\frac{P_r(\sigma)}{r!}$ lies in $\Z_p$,  and hence $\Psi_\ell(\sA, P, Q, v)$ is congruent modulo $\Z_p$ to
\begin{equation} \label{e:usecong}
 \pm \sum_x \sum_{k \in (\ell\Z/\ell p^M \Z)^n}  \sB_\1^{\sL, z}\left(\frac{y + k}{p^M}, \sigma^{-1}Q \right) \sum_r \frac{P_r(\sigma)}{r!}\N\!\left(\frac{y+ k}{\ell}\right)^{r}.
\end{equation}
By the definition (\ref{e:prdef}), the sum over $r$ in (\ref{e:usecong}) is equal to $P(\sigma(y+k) /\ell)$, which by (\ref{e:yprop}) and the given property
\[ P(v + (\textstyle{\frac{1}{\ell}}\Z \oplus \Z^{n-1})) \subset \Z[\frac{1}{\ell}] \] 
 lies in $\Z[\frac{1}{\ell}]$.  Therefore, by Proposition~\ref{p:simplecase}, the quantity in (\ref{e:usecong}) lies in $\frac{1}{m}\Z_p$, and the theorem is proven.
\end{proof}

\begin{proof}[Proof of Proposition~\ref{p:crucial}]
  As in the classical Kubota--Leopoldt construction of $p$-adic $L$-functions over $\Q$, the proof relies on the fact that the Bernoulli polynomial $b_k(x)$ 
  begins \begin{equation} \label{e:bernbegins}
  b_k(x)=x^k-\frac k 2 x^{k-1}+\cdots. 
  \end{equation}
  
  We recall equation  (\ref{equatB}) for $\sB_{\1+r}^{\sL, z}$:
  \begin{equation}
\label{equatB2}
\sB_{\1+r}^{\sL,z}(x,Q)=- \ell^{\underline{r}}\sum_{k=1}^{\l-1} \sum_{y\in  \F_{\!\l}^{\,n}}\textbf{e}\left(\frac{k(\sL(y)-z)}\l\right) \sB_{\1+r}\left(\frac{x+y} \l,Q\right). \end{equation}
  At the expense of altering $z$, we may translate $x$ by an element of $p^M\Z^n$ 
and assume that  $x/p^M$ belongs to $[0, 1)^n.$  Furthermore, for each class in $\F_\ell^n$ we choose the representative $y \in \Z^n$ with $0 \le y_j \le \ell-1$.
   Let $J_0 = \{j \mid x_j = 0 \text{ and } r_j = 0\}$.  For $j \not \in J_0$, (\ref{e:bernbegins}) yields
\[ p^{Mr_j}B_{1 + r_j}\left(\frac{\frac{x_j}{p^M}+y_j}{ \ell}\right)\equiv p^{-M}\left(\frac{x_j}{\ell}\right)^{1 + r_j} +(1 + r_j)\left(\frac{x_j}{\ell}\right)^{r_j}\left(\frac{y_j}\ell-\frac 12\right) \mod p^{M-\epsilon_j} \Z_p, \]
 where $\epsilon_j$ depends only on $r_j,$ $\ell$  and the power of $p$ in the denominator of $x_j.$
Let $a_j \in \F_\ell$ denote the value of $\sL$ on the $j$th standard basis vector of $\Z^n$, and multiply the
previous congruence by $\mbf e(\frac{ka_jy_j}\l)$.  Summing over all $0 \le y_j \le \ell-1$, the leading term of the right side vanishes and 
we obtain
\begin{equation}\label{congruenceBej}
\begin{split}
p^{M r_j} & \sum_{y_j\in \F_\ell}\mbf e\left(\frac{ka_jy_j}\l\right) B_{1+r_j }\left(\frac{\frac{x_j}{p^M}+y_j}{ \ell}\right)\equiv \\
 & (1+r_j)\left(\frac{x_j}{\ell}\right)^{r_j}  \sum_{y_j\in \F_\ell}\mbf e\left(\frac{ka_jy_j}\l\right) B_{1}\left(\frac{\frac{x_j}{p^M}+y_j}{ \ell}\right)\mod p^{M-\epsilon_j} \Z_p[\zeta_\ell].
\end{split}
\end{equation}

 Take the product of these congruences over all $j \not\in J_0$ and multiply by
\begin{equation*}
-\frac {\textbf{e}(-\frac {kz} \l)} m \sum_{i=1}^m \left(\prod_{j\in J_0}\frac{\sign Q_{ij}} 2\right). \end{equation*}
In view of (\ref{equatB2}),
 summing over $k = 1, \dotsc, \ell-1$  gives the desired result (after dividing by $\N(\1 + r)$ and increasing $\epsilon$ accordingly).
\end{proof}

\section{Integrality of smoothed zeta functions} \label{s:integral}

Let $F$ be a totally real field of degree $n$.
In this section we combine the $\Z[1/\ell]$-integrality property of $\Psi_\ell$ proved in Theorem~\ref{t:main} with a suitable generalization of Sczech's formula (\ref{e:special}) to prove Theorem~\ref{t:integral}, recalled below.  Let $\fa$ and $\ff$ be coprime integral ideals of $F$.  Let $\fc$ be an integral ideal of $F$ with norm $\ell$ such that $(\fc, \ff) =1$.  

\begin{oldthm} 
The smoothed zeta function
\[ \zeta_{\ff, \fc}(\fa, s) = \zeta_{\ff}(\fa\fc, s) - \N\fc^{1-s}\zeta_\ff(\fa, s)
\]
assumes values in $\Z[1/\ell]$ at nonpositive integers $s$.
\end{oldthm}

\begin{proof}
Fix a basis $\{w_1, \dotsc, w_n\}$ for $\fa^{-1}\ff$ such that $\{\frac{1}{\ell} w_1, w_2, \dotsc, w_n\}$ is a basis for $\fa^{-1}\fc^{-1}\ff$.  Let $\{\epsilon_1, \dotsc,  \epsilon_{n-1}\}$ denote a basis of the group of totally positive units of $F$ congruent to 1 modulo 
$\ff$.  Let $A_1, \dotsc, A_{n-1}$ be the matrices representing multiplication by the $\epsilon_i$ on $\fa^{-1}\ff$ in terms of the basis 
$w = (w_1, \dotsc, w_n)$, i.e.\ such that
\[ w \epsilon_i = w A_i \qquad i = 1, \dotsc, n. \]
The matrices $A_i$ lie in $\Gamma_{\ell}$, since left multiplication by these matrices preserves the lattice of column vectors
$ \frac{1}{\ell}\Z \oplus \Z^{n-1}  \cong \fa^{-1}\fc^{-1}\ff $ (where the isomorphism is given by dot product with $w$).

Let $P \in \Z[\frac{1}{\ell}][X_1, \dots, X_n]$ denote the homogeneous polynomial of degree $n$ given up to a scalar by the norm:
\begin{equation} \label{e:Pdef2}
 P(X_1, \dotsc, X_n) = \N(\fa\fc)\N(w_1X_1 + \cdots + w_nX_n).
 \end{equation}
 Let $\tilde{Q} = (Q_1, \dotsc, Q_n)$ be the $n$-tuple of linear forms given by
\begin{equation} \label{e:Qdef2} Q_i = \tau_i(w_1^*)X_1 + \cdots + \tau_i(w_n^*)X_n,  \end{equation}
where $\{w_1^*, \dots, w_n^*\}$ is the dual basis  with respect to the trace form on $F$ (i.e.\ $\Tr(w_iw_j^*) = \delta_{ij}$),
and the $\tau_i$ 
denote the embeddings $F \hookrightarrow \R$.
Define the column vector \begin{equation} \label{e:vdef2} v = (\Tr(w_1^*), \dotsc, \Tr(w_n^*)), \end{equation}
so that \[ 1 = v_1 w_1 + v_2w_2 + \cdots + v_nw_n. \]
Dot product with $(w_1, \dotsc, w_n)$ provides a bijection
\begin{equation} \label{e:vinclude}
v + \textstyle\frac{1}{\ell}\Z \oplus \Z^{n-1} \leftrightarrow 1 + \fa^{-1}\fc^{-1}\ff,
\end{equation}
so $P$ and $v$ satisfy the key property
\begin{equation} \label{e:pinclude}
 P(v + \textstyle\frac{1}{\ell}\Z \oplus \Z^{n-1}) \subset \Z[\textstyle{\frac{1}{\ell}}]
\end{equation}
 of Theorem~\ref{t:main}.

The matrices $\{A_i\}$ give rise to a homogeneous $(n-1)$-chain in the standard way, which we write using the bar notation:
$$[A_1\mid \ldots \mid A_{n-1}] =(1, A_1, A_1A_2, \ldots, A_1A_2\cdots A_{n-1})\in \Gamma_{\ell}^n.$$
We symmetrize this chain by defining
\begin{equation} \label{e:cadef}
 \cA(A_1, \dotsc, A_{n-1}) := \sum_{\pi\in S_{n-1}} \sgn(\pi) [A_{\pi(1)}\mid \cdots \mid A_{\pi(n-1)}] \in \Z[\Gamma_{\ell}^n].
\end{equation}
Let 
\begin{equation} \label{e:adef}
 \sA = \rho \cdot \cA(A_1, \dotsc, A_{n-1}).
 \end{equation}
Here $\rho=\pm 1$ is a sign defined as follows: let $\tau_1, \dotsc, \tau_n$ denote the real embeddings of $F$, and consider
the square matrices
\[ W = (\tau_i(w_j))_{i,j=1}^{n}  \quad \text{ and} \quad  
R=(\log \tau_i(\epsilon_{j}))_{i, j=1}^{n-1}. \]
Then $\rho=(-1)^{n-1}\sign(\det W)\sign(\det R).$

For integers $k \ge 0$, Sczech's formula ([Sc2, Corollary p. 595])  reads 
\begin{equation} \label{e:zeta1}
\zeta_{\mf f}(\fa, -k)= \ell^{-k}\Psi(\sA, P^{k},\tilde{Q}, v),
\end{equation}
where the  $\ell$-power term arises from the extra factor of $\N\fc = \ell$ in the definition of $P$.
Applying this formula again with $\fa$ replaced by $\fa\fc$ gives
\begin{equation} \label{e:zeta2}
\zeta_{\mf f}(\fa\fc, -k)= \Psi(\sA', (P')^{k},\tilde{Q}', v'),
\end{equation} 
where \[ \sA' = \pi_\ell A \pi_\ell^{-1}, \quad P' = \pi_\ell^{-1}P, \quad  \tilde{Q}' = \pi_\ell \tilde{Q}, \quad \text{ and } \quad v' = \pi_\ell v \]
as in Section~\ref{s:smoothed}.

 Combining (\ref{e:zeta1}) and (\ref{e:zeta2}), we find
\begin{eqnarray}
\zeta_{\ff, \fc}(\fa, -k) \!\!\! & =  &\!\!\! \zeta_\ff(\fa\fc, -k) - \ell^{1+k}\zeta(\fa, -k). \nonumber \\
& = &\!\!\!  \Psi(\sA', (P')^{k},\tilde{Q}', v') - \ell \Psi(\sA, P^{k},\tilde{Q}, v)  \nonumber \\
& = &\!\!\!  \Psi_\ell^+(\sA, P^k, \tilde{Q}, v). \label{e:zetaaspsi}
\end{eqnarray}
The result $\zeta_{\ff, \fc}(\fa, -k) \in \frac{1}{2n}\Z[1/\ell]$ now follows from Theorem~\ref{t:main}.

The denominator $2n$ may be removed by invoking the following result proven in \cite{CDG}, which 
states that (\ref{e:zetaaspsi}) still holds with $\Psi^+_\ell$ replaced by $\Psi_\ell$, and with $\tilde{Q}$ replaced
by any one of the $n$ individual linear forms defining it:

\begin{theorem}  Let the notation be as above, and let $Q(X_1, \dotsc, X_n) = \sum_{i=1}^{n} \tau(w_i^*) X_i$ for 
a real embedding $\tau: F \hookrightarrow \R$.  Then
\[ \zeta_{\ff, \fc}(\fa, -k) = \Psi_\ell(\sA, P^k, Q , v). \]
\end{theorem}
Now $\zeta_{\ff, \fc}(\fa, -k) \in \Z[1/\ell]$  follows from Theorem~\ref{t:main}.
\end{proof}

\begin{remark}  \label{r:candl}
Note that \[ \zeta_{\ff, \fb\fc}(\fa, s) = \zeta_{\ff, \fc}(\fa\fb, s) + \N\fc^{1-s}\zeta_{\ff, \fb}(\fa, s), \]
so we obtain more generally that $\zeta_{\ff, \fc}(\fa, -k) \in \Z[1/\N\fc]$ for nonnegative integers $k$
when $\fc$ is a product of ideals with prime norm.  Deligne--Ribet show that this result holds for arbitrary integral ideals $\fc$.
Extending our methods to obtain this general result seems difficult, since the fact that the level $\ell$ of our modular group 
$\Gamma_{\ell}$ is squarefree appears at face value to be essential in the argument of Proposition~\ref{p:simplecase}.
\end{remark}

\section{$p$-adic measures and $p$-adic zeta functions} \label{s:padic}

In this section we interpret the cocycle $\Psi_\ell$ in terms of $p$-adic measures and use this perspective to prove Theorem~\ref{t:padic} of the Introduction on the existence of $p$-adic zeta functions.  

\subsection{$p$-adic measures associated to $\Psi_\ell$}

Let $\X = \Z_p^n$, and let 
\[ \Gamma_{\ell, p} := \Gamma_0(\ell \Z[1/p]) = \Gamma_{\ell} \cap \GL_n(\Z[1/p]). \]

Given  $\sA \in \Gamma_{\ell, p}^n$, $Q \in \sQ$, and $v \in \Q^n$, we define a $\frac{1}{m}\Z[1/\ell]$-valued measure 
$\mu_\ell = \mu_\ell(\sA, Q, v)$ on 
\[ \X_v :=  v + \X \subset \Q_p^n \]
as follows.  Let $\sigma$ denote the matrix whose columns are the first columns of the matrices in the tuple $\sA$.  If $\det(\sigma) = 0$, then $\mu_\ell$ is the 0 measure.  Suppose now that $\det(\sigma) \neq 0$.

A vector $a \in  \Z^n$ and a nonnegative integer $r$ give rise to the compact open subset 
\[ a + p^r \X \subset \X.\] 
These sets form a basis of compact open subsets of $\X$, and hence their translates by $v$ form a basis of compact open subsets of $\X_v$.
We define $\mu_\ell$ by applying $\Psi_\ell$ with the constant polynomial $P=1$:
\begin{equation} \label{e:mudef}
 \mu_\ell(\sA, Q, v)(v + a + p^r \X) = \Psi_\ell\left(\sA, 1, Q, \frac{v + a}{p^r}\right) \in \textstyle{\frac{1}{m}}\Z[\textstyle{\frac{1}{\ell}}] \subset \textstyle{\frac{1}{m}}\Z_p. 
 \end{equation}
It is easily checked that this assignment is well-defined, and that the distribution relation for $\Psi_\ell$ yields a corresponding distribution relation for $\mu_\ell$.

Let $\cV = \Q^n/\Z^n$.  Let $\cM_p$ denote the space of functions that assigns to each $(Q, v) \in \sQ \times \cV$ a $\C_p$-valued measure $\alpha(Q, v)$ on $\X_v$ such that $\alpha(Q, pv)(pU) = \alpha(Q, v)(U)$ for all $U \subset \X_v$.
The space $\cM_p$ naturally has the structure of a $\Gamma_{\ell, p}$-module given by
\[ (\gamma \alpha)(Q, v)(U) := \alpha(AQ, Av)(AU), \]
where $A = \lambda \gamma$ is chosen such that $\lambda$ is a power of $p$ and $A \in M_n(\Z)$.
\begin{proposition} \label{p:mucoc}
The function $\mu_\ell: \Gamma_{\ell,p}^n \lra \cM_p$ is a homogenous $(n-1)$-cocycle.
\end{proposition}

Proposition~\ref{p:mucoc} follows directly from the fact that $\Psi_\ell$ is a cocycle. The following theorem shows that the cocycle $\Psi_\ell$
can be recovered from the cocycle of measures $\mu_\ell$; in other words, the cocycle $\Psi_\ell$ specialized to $P=1$ determines its value on all $P \in \sP$.

\begin{theorem}  \label{t:mint} 
For any $P \in \sP$, $a \in \Z^n$, and $M \in M_n(\Z) \cap \Gamma_{\ell, p}$, we have
\begin{equation}
\label{e:psiintM}
 \int_{v+a + M(\X)} P(x) d\mu_\ell(A, Q, v) = \sgn(\det(M)) \cdot \Psi_\ell(M^{-1}A, M^t P, M^{-1}Q, M^{-1}(v+a)).
\end{equation}
In particular
we have
\begin{equation} \label{e:psiint}
\int_{\X_v} P(x) \ d\mu_\ell(\sA, Q, v)(x) =  \Psi_\ell(\sA, P, Q, v).
\end{equation}
\end{theorem}

\begin{proof}
 It suffices to prove the result when $P$ is homogeneous of degree $d$.
 We follow closely the proof of Theorem~\ref{t:main} given in Section~\ref{s:main}.
It was shown there (see (\ref{e:usecong})) that there exists an integer $\epsilon$ such that for each positive integer $N \ge \epsilon$, the quantity \begin{equation}
 \Psi_\ell(M^{-1}\sA, M^tP, M^{-1}Q, M^{-1}(v+a))\end{equation}
  is congruent to
\begin{equation} \pm \sum_x \sum_{k \in (\ell\Z/\ell p^N\Z)^n} \sB_\1^{\sL', -p^{-N}x_1}\left(\frac{y+k}{p^N}, \sigma^{-1}Q\right) P\left(\frac{\sigma (y+k)}{\ell}\right)
\label{e:psilpm} \end{equation}
modulo $p^{N - \epsilon} \Z_p$.  Here $\sigma$ again denotes the matrix of first columns of $\sA$, scaled by an integer relatively prime to $\ell$ such that $M^{-1} \sigma \in M_n(\Z)$.  Meanwhile, $x$ sums over representatives in $\Z^n$ for $\Z^n/\sigma_\ell'\Z^n$, where
$\sigma_\ell' = \pi_{\ell} \ell^{-1} M^{-1} \sigma\in M_n(\Z)$,  and $y = (\sigma_\ell')^{-1}x + \sigma^{-1}\ell(v+a)$.  The $\pm$ sign is $(-1)^n \sign(\det M^{-1} \sigma).$   Finally, the linear form $L'$ is given by $L'(y) = \sigma_\ell'(y)_1 = $ the inner product of $y$ with the first row of $M^{-1}\sigma$.
The expression (\ref{e:psilpm}) is simplified with a change of variables.  
First replace the variable $k$ by arbitrary representatives $j$ for $\Z^n/p^N\Z^n$ (not necessarily divisible by $\ell$) such that $j \equiv k \pmod{p^N}$;
the expression $\sB_\1^{\sL', -p^{-N}x_1}\left(\frac{y+k}{p^N}, \sigma^{-1}Q\right)$ is seen from the definitions to equal
$\sB_\1^{\sL', -p^{-N}(x_1 + \sL'(j))}\left(\frac{y+j}{p^N}, \sigma^{-1}Q\right)$.
Then let $u = x + \sigma_\ell'(j)$; by Lemma~\ref{l:bijlem} the expression  (\ref{e:psilpm}) is congruent modulo $p^N$ to:
\begin{equation} \label{e:uchange}
 \pm \!\!\!\!\!\! \sum_{u \in \Z^n/p^N \sigma_\ell' \Z^n} \sB_{\1}^{\sL', -p^{-N}u_1} \left( \frac{(\sigma_\ell')^{-1}(u) + \sigma^{-1}\ell (v+a)}{p^N}, \sigma^{-1}Q \right)P(M\pi_\ell^{-1} u + v + a). 
 \end{equation}

Let us meanwhile evaluate the Riemann sums approximating the integral 
on the right side of (\ref{e:psiintM}). 
There is a $\delta$ depending on the powers of $p$ in the denominator of $P(v)$ such that for $N$ large we have
\begin{equation} \label{e:mintcong}
 \int_{v + a + M\X} P(x) \ d\mu_\ell(x) \equiv   \sum_{h \in \Z^n/p^N \Z^n} P(v + a + Mh) \mu_\ell(v + a + Mh + p^N M\X) \pmod{p^{N-\delta}\Z_p} 
 \end{equation}
 Let $r$ be large enough that $p^r \X \subset M\X$.  We can then apply the definition of $\mu_\ell$ using the decomposition
 \[ v + a+ Mh + Mp^N\X = \bigsqcup_{s \in M\Z^n / p^r \Z^n} (v + a +Mh + sp^N + p^{N+r}\X).
 \]
 Using the change of variables $j = Mh + sp^N$, we obtain  by Lemma~\ref{l:bijlem} that the right side of (\ref{e:mintcong}) is equal to
 \begin{equation}
\pm  \sum_{j \in M\Z^n/p^{N+r} \Z^n} P(v + a + j) \cD_{\ell}\!\left(\sigma, \1,Q, \frac{v + a + j}{p^{N+r}} \right), \label{e:pintM}
\end{equation}
where the $\pm$ sign is $(-1)^n \sign(\det \sigma)$.
We rename $N+r$ as $N$ for simplicity, since $r$ is fixed and we will be taking $N \rightarrow \infty$.
Applying Lemma~\ref{l:decompose} for the occurrence of $\cD_\ell$ in (\ref{e:pintM}), we obtain
 \begin{equation}
\pm   \sum_{j \in M\Z^n/p^{N} \Z^n} \sum_{z \in \Z^n/\sigma_\ell \Z^n} P(v + a + j) \sB_1^{L, - z_1}(\sigma_\ell^{-1}(z + \pi_\ell \left(\frac{v + a + j}{p^{N}}\right)), \sigma^{-1} Q). \label{e:pintM2}
   \end{equation}
 Now if we fix representatives $\{j \}$ and $\{ z\}$ for 
 $M\Z^n/p^{N} \Z^n$ and $\Z^n/\sigma_\ell \Z^n$ respectively, then by Lemma~\ref{l:bijlem} the map $(j, z) \mapsto  \pi_\ell M^{-1} j + p^N \pi_\ell M^{-1} \pi_\ell^{-1} z$
 gives a bijection
 \[ M\Z^n/p^{N} \Z^n  \times \Z^n/\sigma_\ell \Z^n \longleftrightarrow \Z^n/ p^N \sigma_\ell' \Z^n. 
 \]
 The change of variables $u = \pi_\ell M^{-1} j + p^N \pi_\ell M^{-1} \pi_\ell^{-1} z$ then shows that the expressions 
 (\ref{e:uchange}) and (\ref{e:pintM2}) are congruent modulo $p^{N - \max(\delta, \epsilon)}$ up to the discrepancies in the $\pm$ signs, which is
 $\sign(\det M)$.  Here we use the fact that 
 \[ L'(y) \equiv - p^{-N} u_1 \Leftrightarrow L(y) \equiv - p^{-N} M_{11} u_1 \Leftrightarrow L(y) \equiv - z_1 \pmod{\ell}. \]
 Letting $N \rightarrow \infty$ gives the desired result.

\end{proof}

\subsection{$p$-adic zeta functions}



We now return to the setting of a totally real field $F$.  Let $p$ be a rational prime.
Let $\fa$ be an integral ideal of $F$ and let $\fc$ be an ideal of norm $\ell$ such that $(\fa\fc, \ff)=1$.
We will use the $p$-adic measures defined above to construct the $p$-adic zeta function 
$\zeta_{\ff, \fc, p}(\fa, s)$ of Theorem~\ref{t:padic}, which we recall below.

\begin{oldthm}  There exists a unique $\Z_p$-valued analytic function
$\zeta_{\ff, \fc, p}(\fa, s)$ of the $p$-adic variable $s \in \cW$ such that
\begin{equation} \label{e:interp}
 \zeta_{\ff, \fc, p}(\fa,  - k) = \zeta^*_{\ff, \fc}(\fa,  -k)
\end{equation}
for all nonnegative integers $k$.
\end{oldthm}

\begin{proof}
First we note that it suffices to consider the case where $\ff$ is divisible by all primes of $F$ above $p$. 
Indeed, if we let $\fg$ denote the least common multiple of $\ff$ and the primes above $p$, then we can define the $p$-adic zeta functions
attached to $\ff$ from the ones attached to $\fg$ as follows:
\[ \zeta_{\ff, \fc, p}(\fa, s) = \sum_{\fb \sim_\ff \fa} \zeta_{\fg, \fc, p}(\fb, s).
\]
Here the sum ranges over representatives $\fb$ for the narrow ideal class group $G_\fg$ whose images in $G_\ff$ are equivalent to $\fa$.
The analogous equation for the classical partial zeta functions $\zeta^*$ follows from the fact that the sum defining $\zeta^*$ ranges over ideals
relatively prime to $p$.

Therefore, suppose that $\ff$ is divisible by all primes of $F$ above $p$.  Then $\zeta^*_{\ff, \fc} = \zeta_{\ff, \fc}$.
Let $P, Q, v$, and $\sA$ be as in Section~\ref{s:integral}, 
and let $\mu_\ell = \mu_\ell(\sA, Q, v)$ as above.
As we saw in (\ref{e:vinclude})--(\ref{e:pinclude}), for any $x \in \Z^{n}$ the quantity $P(v + x)$ is the norm of an integral ideal of $F$ relatively prime to $\ff$.  Since $\ff$ is divisible by all the primes above $p$, this quantity is an integer relatively prime to $p$.  By the continuity of $P$, we find that \[ P(v+x) \in \Z_p^\times \text{ for all } x \in \X, \]  i.e. $P(\X_v) \subset \Z_p^\times$.
We may therefore define a $p$-adic analytic function on $\cW$:
\begin{equation} \label{e:zetapdef}
 \zeta_{\ff, \fc, p}(\fa, s) :=   \int_{\X_v} P(x)^{-s} \ d\mu_\ell(A, Q, v).
\end{equation}
Here we have followed the usual convention of writing $x^{-s}$ for $s(x)^{-1}$ when $x \in \Z_p^\times, s \in \cW$.

For a nonnegative integer $k$, equation (\ref{e:zetaaspsi}) and Theorem~\ref{t:mint} yield
\begin{align}
\zeta_{\ff, \fc}(\fa, -k) =& \  \Psi_\ell(\sA, P^k, Q, v) \nonumber \\
=&  \int_{\X_v} P(x)^{k} \ d\mu_\ell(A, Q, v). \label{e:zetaint}
\end{align}
Equation (\ref{e:zetaint}) gives the desired interpolation property (\ref{e:interp}).
\end{proof}

With our applications to Gross's Conjecture~\ref{c:gross} in mind, it is useful to have a formula such as (\ref{e:zetapdef}) for the $p$-adic zeta-function 
when $\ff$ is not necessarily divisible by the primes above $p$. 
Write $\ff = \ff_0 \ff_1$, where $\ff_0$ is the prime-to-$p$ part of $\ff$ and $\ff_1$ is divisible only by primes above $p$.
Let $\ff_p$ denote the product of the primes dividing $p$ that do not divide $\ff$.  Fix an integral ideal $\fa$ relatively prime to $p$. 
An elementary calculation following directly from the definitions shows that
for $s \in \C$, 
\begin{equation} \label{e:zetastar} \zeta_{\ff, \fc}^*(\fa, s) = \sum_{\fb \mid \ff_p} \mu(\fb)\N\fb^{-s} \zeta_{\ff, \fc}(\fa\fb^{-1}, s),
\end{equation}
where $\mu(\fb) = \pm 1$ is determined by the parity of the number of prime factors of $\fb$.  For integers $s \le 0$, we will express each term of (\ref{e:zetastar}) as an integral
with respect to the measure $\mu_\ell$ in such a way that the  sum can be interpolated $p$-adically.

 Define the variables $A, P, Q, v$ as in Section~\ref{s:integral} using the ideals $\fa$ and $\ff_0$.  In particular,
 $\{w_i\}$  is a $\Z$-basis of $\fa^{-1}\ff_0$.
 Dot product with 
$w = (w_1, \dots, w_n)$ gives a bijection between the spaces $\X_v = \X$ and $\cO_p = \prod_{\fp \mid p} \cO_{\fp}$.   Using this bijection we view $\mu_\ell(A, Q, v)$ as a measure on $\cO_p$.
For each prime ideal $\fp \mid p$, define $\cO_{\fp, \ff}: = 1 + \ff \cO_\fp$, and write
\[  \cO_{p, \ff} :=1 + \ff \cO_p =  \prod_{\fp \mid p} \cO_{\fp, \ff}.\]
Let $\fb$ denote an integral ideal of $F$ with $p$-power norm such that $(\fb, \ff)=1$ (so $\fb$ is a product of prime ideals dividing $\ff_p$).  Define
\[ \cO_{p, \fb, \ff} := \fb \cO_p \cap \cO_{p, \ff} = \prod_{\fp | \fb} \fb \cO_{\fp} \times \prod_{\fp \mid p, \ \fp \nmid \fb} \cO_{\fp, \ff} . \]

Write \[ 
\cO_{\fp , \ff}^* := \cO_\fp^* \cap \cO_{p , \ff}, \qquad
\cO_{p , \ff}^* := \prod_{\fp \mid p} \cO_{\fp , \ff}^*, \qquad
 \cO_{p ,\fb, \ff}^* := \prod_{\fp \mid p} \fb \cO_\fp^* \cap \cO_{p , \ff}. \]
  

The following formula generalizes (\ref{e:zetapdef}) to the current setting, where $\ff$ is not necessarily divisible by all primes above $p$.

\begin{proposition} \label{p:zetapb}
The $p$-adic zeta-function of Theorem~$\ref{t:padic}$ is given explicitly by the following integral representation:
\begin{equation} \label{e:zintgen} \zeta_{\ff, \fc, p}(\fa, s) = (\N\fa\fc)^{-s} \int_{\cO_{p , \ff}^*} (\N x)^{-s} \ d\mu_\ell(A, Q, v). 
\end{equation}
More generally, with $\fb$ as above we have
\begin{equation} \label{e:zintgenb}
 \zeta_{\ff, \fc, p}(\fa\fb^{-1}, s) = (\N\fa\fc)^{-s} \int_{\cO_{p , \fb, \ff}^*} (\N x)_p^{-s} \ d\mu_\ell(A, Q, v), 
 \end{equation}
 where  $x_p := x/p^{\ord_p(x)}$ is the unit part of $x \in \Q_p^\times$.
\end{proposition}

\begin{proof}
We first express the zeta value $\zeta_{\ff, \fc}(\fa \fb^{-1}, -k)$ for a nonnegative integer $k$ as an integral over the space $\cO_{p, \fb, \ff}$.
Fix an integral ideal $\fq$ relatively prime to $\ff p$ whose image in $G_\ff$ is the inverse class of $\fb$.  Therefore $\fb \fq = (\pi)$ for a totally positive $\pi \equiv 1 \pmod{\ff}$.

Fix a $\Z$-basis $u = (u_1, \dots, u_n)$ for $\fa^{-1}\fb\ff$ such that $(\frac{1}{\ell}u_1, u_2, \dots, u_n)$ is a basis for 
$\fa^{-1}\fc^{-1}\fb\ff$.  Let $M$ be the matrix such that $wM = u$.  It is clear that $M \in M_n(\Z)$, and that $| \det M | = \N\fb\ff_1$ is a  power of $p$.   Furthermore, it is clear that  $M \in \Gamma_{\ell, p}$, so $M$ satisfies the conditions of Theorem~\ref{t:mint}.
 For simplicity let us choose the basis $u$ such that $\det M > 0$.

Next, let $R$ note the matrix representing multiplication by $\pi$ with respect to the basis $w$, i.e.\ such that $w \pi = w R$.  Note that $R$ commutes with the matrices $A_i$ in the definition of the chain $A$.  Now $\fa^{-1} \fq^{-1} \ff = \pi^{-1}(\fa^{-1} \fb \ff)$ has basis $\pi^{-1} u = w R^{-1} M$.  If we use this basis to define the variables $A_{\fa \fq}, P_{\fa\fq}, Q_{\fa\fq}, v_{\fa\fq}$ as in Section~\ref{s:integral} using  the ideals $\fa \fq$ and $\ff$, then we find
\begin{align*}
 A_{\fa \fq} &= M^{-1} A M, \\ P_{\fa \fq}  &= \N\fb^{-1} \cdot (M^t P), \\ Q_{\fa \fq} &\sim M^{-1}Q,
   \\ v_{\fa \fq} &\equiv M^{-1}(v + a) \pmod{\Z^n},
\end{align*} 
where $a \in \Z^n$ is chosen so that $w \cdot a = \pi - 1 \in \fa^{-1} \ff_0$.
Here $Q_{\fa\fq} \sim M^{-1}Q$ means that the corresponding tuples of linear forms are equal up to 
 up to scaling the linear forms by positive reals (which does not affect the value of the cocycle $\Psi_\ell$).


For any nonnegative integer $k$, we therefore find
\begin{align}
\zeta_{\ff, \fc}(\fa\fq, -k) &= \  \Psi_\ell(\sA_{\fa\fq} , P_{\fa\fq}^k, Q_{\fa\fq}, v_{\fa\fq}) \nonumber \\
&= \ \N\fb^{-k} \Psi_\ell(M^{-1}AM , (M^t P)^{k}, M^{-1} Q, M^{-1}(v+a)) \nonumber \\
&= \ \N\fb^{-k} \Psi_\ell(M^{-1}A , (M^t P)^{k}, M^{-1} Q, M^{-1}(v+a)) \label{e:sclemma} \\
&=  \N\fb^{-k} \int_{v + a + M(\X)} P(x)^{k} \ d\mu_\ell(A, Q, v) \label{e:usemint} \\
&= \ (\N\fa\fc\fb^{-1})^k \int_{\cO_{p, \fb , \ff}} (\N x)^{k} \ d\mu_\ell(A, Q, v). \label{e:transfer}
\end{align}

Equation (\ref{e:sclemma}) follows from \cite[Lemma 4]{Sc2}.  Equation (\ref{e:usemint}) follows from Theorem~\ref{t:mint}.  In equation (\ref{e:transfer}) we have identified
$ v + a + M(\X)$ with $ \cO_{p, \fb , \ff}$ via dot product with $w$.
  From (\ref{e:zetastar}) and (\ref{e:transfer}) and an inclusion-exclusion argument, it follows that
\[ \zeta_{\ff, \fc}^*(\fa, -k) =   (\N\fa\fc)^k  \int_{\cO_{p , \ff}^*} (\N x)^{k} \ d\mu_\ell(A, Q, v). \]
This proves (\ref{e:zintgen}) by interpolation, and (\ref{e:zintgenb}) holds similarly.
\end{proof}

\section{Order of Vanishing at $s=0$} \label{s:oov}

\subsection{$p$-adic $L$-functions} \label{s:oovpadic}

As in the introduction, 
let $\chi: \Gal(\overline{F}/F) \rightarrow \overline{\Q}^*$ be a totally odd finite order character with conductor $\ff$.  The $p$-adic $L$-function  $L_{\fc,p}(\chi\omega, s)\colon \Z_p \longrightarrow \C_p^*$ is given by
\[
 L_{\fc, p}(\chi \omega, s) =  \sum_{\fa \in G_\ff} \chi(\fa\fc) \zeta_{\ff, \fc, p}(\fa, \langle \cdot \rangle^s).
 \] 
Let $\fp_1, \dotsc, \fp_{r_\chi}$ denote the primes above $p$ such that $\chi(\fp_i) = 1$.  (In particular each $\fp_i \nmid \ff$.)  The goal of the rest of the paper is to prove Theorem~\ref{t:oov}, which states that
\[
 L_{\fc, p}^{(k)}(\chi \omega, 0) = 0  \text{ for } k < r_\chi.
 \] 
In the sequel, we write simply $r$ for $r_\chi$.
Let $G_{p, \chi} \subset G_\ff$ denote the subgroup generated by the images of $\fp_1, \dotsc, \fp_r$.  
Let $e_i$ denote the order of $\fp_i$ in $G_\ff$, and write $\fp_i^{e_i} = (\pi_i)$ for a totally positive $\pi_i \equiv 1 \pmod{\ff}$.
Let \[ e = \left(\prod_{i}^{r} e_i\right)/\# G_{p, \chi} \in \Z. \]

We then have
\begin{align}
 L_{\fc, p}(\chi \omega, s) &= \sum_{\fa \in G_\ff / G_{p, \chi}} \chi(\fa\fc) \sum_{\fb \in G_{p, \chi}} \zeta_{\ff, \fc, p}(\fa\fb^{-1}, s) \nonumber \\
&= \frac{1}{e} \sum_{\fa \in G_\ff / G_{p, \chi}} \chi(\fa\fc) \sum_{\fb \mid \prod_{i=1}^{r} \fp_i^{e_i-1}} \zeta_{\ff, \fc, p}(\fa\fb^{-1}, s) \nonumber \\
&=  \frac{1}{e} \sum_{\fa \in G_\ff / G_{p, \chi}} \chi(\fa\fc) \langle \N\fa\fc \rangle^{-s} \int_{\O} \langle \N x \rangle^{-s} d\mu_{\ell}(A_\fa, Q_\fa, v_\fa), \label{e:oint}
\end{align}
where 
\[ \O = \prod_{i=1}^{r} (\cO_{\fp_i} - \pi_i \cO_{\fp_i}) \times \prod_{\fp \mid p, \ \fp \neq \fp_i} \cO_{\fp, \ff}^*. \]
Equation (\ref{e:oint}) follows from Proposition~\ref{p:zetapb}.  (As usual, the representative ideals $\fa$ are chosen relatively prime to $\ff p$.)

In order to prove Theorem~\ref{t:oov}, it therefore suffices to show that the integral in (\ref{e:oint}) has order of vanishing at least $r$, i.e. that
\begin{equation} \label{e:oovlog}
 \int_{\O} (\log_p  \N x )^k d\mu_{\ell}(A, Q, v) = 0 \text{ for } 0 \le k < r.
\end{equation}

\subsection{Spiess' Theorems and the proof of Theorem~\ref{t:oov}}

In this section, we explain how the cocycle of measures $\mu_\ell$ can be combined with Spiess's 
cohomological formalism for $p$-adic $L$-functions to deduce (\ref{e:oovlog}), and thereby prove Theorem~\ref{t:oov}.  All of the definitions, results, and proofs in this section are due to 
Spiess \cite{sphmf}.

Denote by $E \subset \cO_F^\ast$ the group of totally positive units of $F$ congruent to 1 modulo $\ff$.  Denote by $T$ the subgroup of $F^\ast$ generated by the $\pi_i$, for $i = 1, \dotsc, r$.
Denote by $U \cong E \times T$ the subgroup of $F^\ast$ generated by $E$ and $T$.  

Write $F_p = \prod_{\fp \mid p} F_\fp$ for the completion of $F$ at $p$.
Let $C_c(F_p)$ denote the $\C_p$-algebra of $\C_p$-valued continuous functions on $F_p$
with compact support, and similarly for $C_c(F_\fp)$, for $\fp \mid p$.

 Note that $U$ acts on $C_c(F_p)$ by $(u \cdot f)(x) := f(x/u)$.
The cocycle $\mu_\ell$ along with the ideal $\fa$ allow for the definition of a homogeneous cocycle 
\[ \kappa_\fa \in Z^{n-1}(U, C_c(F_p)^\vee) \] as follows.  Given $\epsilon_1, \dots, \epsilon_n \in U$, let $A_i \in \Gamma_{\ell, p}$ denote the matrix for multiplication by $\epsilon_i$ with respect to the basis $\{w_i\}$ of $\fa^{-1}\ff_0$, and define
\[ \kappa_\fa(\epsilon_1, \dotsc, \epsilon_n)(f) = \int_{F_p} f(x) d\mu_\ell(A_1, \dotsc, A_n)(Q_\fa, v_\fa). \]
 As in the previous section, we have identified $\Q_p^n$ with $F_p$ via dot product with $(w_1, \dotsc, w_n)$, and thereby view $\mu_\ell$ as a compactly supported measure on $F_p$. 
 (Initially $\mu_\ell$ was defined on a compact subset of $\Q_p^n$, and we extend it by zero to a compactly supported measure on $\Q_p^n$.) 
The cocycle $\kappa_\fa$ yields a class $[\kappa_\fa] \in H^{n-1}(U, C_c(F_p)^\vee)$.

Meanwhile, for each nonnegative integer $k$ 
we define a class $[\fL_k] \in H_{n-1}(U, C_c(F_p))$ as follows.
Let $\epsilon_1, \dotsc, \epsilon_{n-1}$ denote a basis of $E$, and define
\[ \fL_k := \cA(\epsilon_1, \dotsc, \epsilon_{n-1}) \otimes \mathbf{1}_\O \cdot (\log_p \N x)^k \in
Z_{n-1}(U, C_c(F_p)), \]
where $\cA$ is defined as in (\ref{e:cadef}).
Tracing through these notations, it is clear that the left side of (\ref{e:oovlog}) is given by
\begin{equation} \label{e:intpair}
 \int_{\O} (\log_p  \N x )^k d\mu_{\ell}(A, Q, v)  = \pm \langle [\kappa_\fa], [\fL_k] \rangle, 
 \end{equation}
where the pairing on the right is the usual cap product
\[ H^{n-1}(U, C_c(F_p)^\vee) \times H_{n-1}(U , C_c(F_p)) \longrightarrow \C_p \]
(and the $\pm$ is given by the sign $\rho$ appearing in the definition (\ref{e:adef}) of $A$).
In view of the discussion of Section~\ref{s:oovpadic} (in particular (\ref{e:oovlog})) and (\ref{e:intpair}), Theorem~\ref{t:oov} will follow if we can prove that \begin{equation} \label{e:cyczero}
 [\fL_k] = 0 \text{ in } H_{n-1}(U , C_c(F_p)) \text{ for } k < r. 
 \end{equation}

The functions $ \mathbf{1}_\O \cdot (\log_p \N x)^k$ lie in a certain subspace
$C_c^\flat(F_p) \subset C_c(F_p)$ that we now define.
Write $S = \{ \fp_1, \dotsc, \fp_r\}$, and for each $\fp \in S$ define
$C_c^\flat(F_\fp)$ to be the subspace of $C_c(F_\fp)$ consisting of those functions that are constant
in a neighborhood of $0$.  Define
\[ C_c^\flat(F_p) := \bigotimes_{\fp \in S} C_c^\flat(F_\fp) \otimes \bigotimes_{\fp \mid p, \ \fp \not\in S} C(\cO_\fp^\ast) \subset C_c(F_p). \]
In other words, $C_c^\flat(F_p)$ consists of functions on $F_p$ that can be written as finite $\C_p$-linear combinations of products $\prod_{\fp \mid p} f_\fp$, with $f_\fp \in C_c^\flat(F_\fp)$ for $\fp \in S$ and $f_\fp \in C(\cO_\fp^\ast)$ for $\fp \mid p$, $\fp \not\in S$. 
 Note that $ \mathbf{1}_\O \cdot (\log_p \N x)^k \in C_c^\flat(F_p)$, since
 for a tuple $x = (x_\fp) \in F_p$ with $x_\fp \in \cO_\fp$, $x_\fp \neq 0$,
we have 
\begin{equation} \label{e:lpsum}
\log_p \N x = \sum_{\fp} \ell_\fp(x), \text{ where } \ell_\fp(x) := \log_p \N_{F_\fp/\Q_p} x_\fp. 
\end{equation}

We will show that in fact
\begin{equation} \label{e:cyczero2}
 [\fL_k] = 0 \text{ in } H_{n-1}(U , C_c^\flat(F_p)) \text{ for } k < r,
 \end{equation}
 which of course implies (\ref{e:cyczero}).
 The proof of (\ref{e:cyczero2})
is broken into two steps. Let $I$ denote the augmentation ideal of the group ring $\C_p[T]$. 

\begin{theorem}[Spiess, \cite{sphmf}]  \label{t:sp1} The natural map $H_{n-1}(U, C_c^\flat(F_p)) \lra H_{n-1}(E, C_c^\flat(F_p)/I)$ is an isomorphism.
\end{theorem}

\begin{theorem}[Spiess, \cite{sphmf}] \label{t:sp2} For $k < r$, we have \[ \mathbf{1}_\O \cdot (\log_p \N x)^k \in I \cdot C_c^\flat(F_p),\] and in particular
the image of $[\fL_k]$ in $H_{n-1}(E, C_c^\flat(F_p)/I)$ vanishes.
\end{theorem}

Theorems~\ref{t:sp1} and \ref{t:sp2} combine to yield (\ref{e:cyczero2}), which in turn combined with (\ref{e:intpair}) and (\ref{e:oint}) yields the proof of Theorem~\ref{t:oov}. 
For completeness, we recall Spiess's proofs of these results.

\begin{lemma}  \label{l:free}
The space $C_c^\flat(F_p)$ is a free $\C_p[T]$-module.
\end{lemma}

\begin{proof}  
We show by induction on $r$ that 
\[ C_c^\flat(F_S) :=  \bigotimes_{\fp \in S} C_c^\flat(F_\fp) \]
is a free $T$-module.  To this end, fix $\fp = \fp_r \in S$ and write $\pi = \pi_r$, $S' = \{\fp_1, \dotsc, \fp_{r-1}\}$, and 
$T' = \langle \pi_i \rangle_{i=1}^{r-1}$, so $T = T' \times \langle \pi \rangle$.
Our inductive hypothesis is that $C_c^\flat(F_{S'})$ is free as a $\C_p[T']$-module.

The space $C_c(F_\fp^\ast)$ of compactly supported continuous functions on $F_\fp^\ast$ can
be identified with the subspace of $C_c^\flat(F_\fp)$ consisting of those functions that vanish on a neighborhood of 0. 
 Write $C_c^0(F_\fp)$ for the space of compactly supported locally constant functions on $F_\fp$, so in particular 
$C_c^\flat(F_\fp)$ is generated by its subspaces $C_c(F_\fp^\ast)$ and $C_c^0(F_\fp)$, and we have
\[ C_c^\flat(F_\fp) / C_c^0(F_\fp) \cong C_c(F_\fp^\ast) / C_c^0(F_\fp^\ast). \]
We therefore obtain an exact sequence
\begin{equation} \label{e:cseq}
0 \longrightarrow C_c^\flat(F_{S'}) \otimes C_c^0(F_\fp) \longrightarrow C_c^\flat(F_S)
\longrightarrow C_c^\flat(F_{S'}) \otimes C_c(F_\fp^\ast) / C_c^0(F_\fp^\ast) \longrightarrow 0.
\end{equation}
It suffices to prove that the first and third terms of the sequence are free $\C_p[T]$-modules.

Since $\cF_\fp := \cO_\fp - \pi \cO_\fp$ is a fundamental domain for the action of $\pi$ on $F_\fp^\ast$, we have
\begin{equation}
 C_c^0(F_\fp) =  C_c^0(F_\fp^\ast) \oplus \C_p 1_{\cO_\fp} =  (\Ind^{\langle \pi \rangle} C^0(\cF_\fp)) \oplus \C_p 1_{\cO_\fp} 
 \label{e:cdecomp}
 \end{equation}
  as $\C_p[T']$-modules.
  Choose a $T'$-stable decomposition
$ C^0(\cF_\fp) = V \oplus \C_p 1_{\cF_\fp}$ (for instance, we may take $V \subset C^0(\cF_\fp)$ to be the subspace of functions that have integral against Haar measure on $\cO_\fp$ equal to 0).
Using $1_{\cF_\fp} = (1 - \pi)1_{\cO_\fp}$, one sees from (\ref{e:cdecomp}) that 
\[ C_c^0(F_\fp) = \Ind^{\langle \pi \rangle} (V \oplus 1_{\cO_\fp}) \]
as $\C_p[T]$-modules.  Lemma~\ref{l:g1g2} below then implies that 
$ C_c^\flat(F_{S'}) \otimes C_c^0(F_\fp)$ is a free $\C_p[T]$-module.

Similarly,
\[ C_c(F_\fp^\ast) / C_c^0(F_\fp^\ast) \cong \Ind^{\langle \pi \rangle}(C(\cF_\fp)/C^0(\cF_\fp)) \]
as a $\C_p[T]$-module, and hence the third term of (\ref{e:cseq}) is a free $\C_p[T]$-module
by  Lemma~\ref{l:g1g2}.  We conclude that $C_c^\flat(F_S)$ is a free $\C_p[T]$-module,
and hence $C_c^\flat(F_p)$ is as well.
\end{proof}

\begin{lemma} \label{l:g1g2}
Let $G_1$ and $G_2$ be groups, $G = G_1 \times G_2$, and let $K$ be a field.
Let $M_1, M_2$ two $K[G]$-modules such that $M_2 \cong \Ind^{G_2} N_2$ as a $K[G]$-module for some $G_1$-module $N_2$, and $M_1 \cong \Ind^{G_1} N_1$ as $K[G_1]$-modules for some $K$-vector space $N_1$.  
Then $M_1 \otimes_K M_2$ is a free $K[G]$-module.
\end{lemma}

\begin{proof}
One verifies that $M_1 \otimes_K M_2 \cong \Ind^G (N_1 \otimes_K \Res^{G_1}_{1} N_2).$
\end{proof}

Lemma~\ref{l:free} immediately yields:

\begin{proof}[Proof of Theorem~\ref{t:sp1}]
Since $C_c^\flat(F_p)$ is a free $\C_p[T]$-module, it has homological dimension 1, and
the Hochschild--Serre spectral sequence
\[ E^2_{p,q} = H_p(E, H_q(T, C_c^\flat(F_p))) \Longrightarrow H_{p+q}(U, C_c^\flat(F_p))
\]
degenerates at $E^2$.  Therefore the maps $E_n \longrightarrow E^2_{n,0}$ are isomorphisms as desired.
\end{proof}

We now move on to:

\begin{proof}[Proof of Theorem~\ref{t:sp2}]  As above we write $S = \{ \fp_1, \dotsc, \fp_r \}$. For any subset $R \subset S$, let
\[ \O_R :=  \prod_{\fp_i \in R} (\cO_{\fp_i} - \pi_i \cO_{\fp_i}) \times
\prod_{\fp_i \in S- R} \cO_{\fp_i} \times 
 \prod_{\fp \mid p, \ \fp \not\in S } \cO_{\fp, \ff}^*,
\]
so in particular $\O_S = \O$.  

From (\ref{e:lpsum}) it follows that $(\log_p \N x)^k$ can be expanded as a sum of monomials of the form 
\[ \ell(x)^{n} := \prod_{\fp | p} \ell_\fp(x)^{n_\fp},  \]
 where $n = (n_\fp)_{\fp \mid p}$ is a tuple of nonnegative integers such that $|n| := \sum_{\fp} {n_\fp} = k$.
We will prove by induction on $|n|$ that if $R \subset S$ is a subset such that $n_\fp = 0$ for all $\fp \in S - R$ and
$|R| > |n|$, then
\[ 1_{\O_R} \cdot \ell(x)^n \in I \cdot C_c^\flat(F_p).
\]
The desired result will then follow from the case $R =S$, $|n| = k$.

For the base case, $|n| = 0$, we choose $\fp_i \in R$ (which is possible since $|R| > 0$), 
let $R' = R - \{ \fp_i\}$, and note
\begin{equation} \label{e:charlin}
1_{\O_R} = (1 - \pi_i) \cdot 1_{\O_{R'}}. 
\end{equation}

The inductive step is similar, using the linearity of the functions $\ell_\fp$.
Given $R \subset S$ such that $n_\fp = 0$ for all $\fp \in S - R$ and
$|R| > |n|$, we may choose a $\fp_i \in R$ such that $n_{\fp_i} = 0$.  
 Let $R' = R - \{\fp_i \}$.  Then:
\begin{align}
(1 - \pi_i) \ell(x)^n =&\  \prod_{\fp \mid p} \ell_\fp(x)^{n_\fp} - \prod_{\fp \mid p} (\ell_\fp(x) - \ell_\fp(\pi_i))^{n_\fp} \nonumber \\
= & \sum_{n', |n'| < |n|} a_{n'} \ell(x)^{n'} \label{e:loglin}
\end{align}
for some coefficients $a_{n'}$ (where $n_\fp' = 0$ for all $\fp \in S - R'$ if $a_{n'} \neq 0$).
Using the formula
\[ ((1 - \pi)f) \cdot g = (1-\pi)(f \cdot g) - f \cdot ((1 - \pi)g) + ((1 - \pi)f)\cdot((1-\pi)g) \]
applied with $f = 1_{\O_{R'}}$, $g = \ell(x)^n$, and $\pi = \pi_i$, equations (\ref{e:charlin}) and 
(\ref{e:loglin})
combine to give
\begin{align*}
1_{\O_R} \cdot \ell(x)^n =& ((1 - \pi_i) 1_{\O_{R'}}) \cdot \ell(x)^n \\
\equiv & - 1_{\O_{R'}} \cdot \sum_{n'} a_{n'} \ell(x)^{n'} + 1_{\O_R} \cdot \sum_{n'} a_{n'} \ell(x)^{n'} \pmod{I \cdot C_c^\flat(F_p)}.
\end{align*}
Each of the terms $1_{\O_{R'}} \cdot \ell(x)^{n'}$ and $1_{\O_R} \cdot \ell(x)^{n'}$ lies in $I \cdot C_c^\flat(F_p)$ by the inductive hypothesis, yielding $1_{\O_R} \cdot \ell(x)^n \in I \cdot C_c^\flat(F_p)$ as desired.
\end{proof}

\bigskip
\bigskip

P.~C. :  \textsc{Institut de Math\'ematiques de Jussieu,
Universit\'e Paris 6,  France} 

\textit{E-mail} \textbf{pierre (dot) charollois (at) imj-prg  (dot) fr}

\medskip

S.~D. : \textsc{Dept. of Mathematics, University of California  Santa Cruz, USA} 

\textit{E-mail} \textbf{sdasgup2 (at) ucsc (dot) edu}

\end{document}